\documentclass[11pt,twoside]{amsart}

\numberwithin{equation}{section}
\usepackage{latexsym,amssymb,amsmath,amsthm,amsfonts,amscd}
\usepackage{enumerate}
\usepackage[table]{xcolor}
\usepackage{graphicx}
\usepackage{tikz}
\usetikzlibrary{matrix,arrows}
\usepackage{multicol}
\usepackage{multirow}

\definecolor{VerdeOlivo}{rgb}{0.3,0.5,0.1}
\definecolor{Magenta}{rgb}{.65,0.15,.2}
\definecolor{Gris}{gray}{0.3}

\newtheorem{Theorem}{Theorem}[section] 
\newtheorem{Definition}[Theorem]{Definition}
\newtheorem{Proposition}[Theorem]{Proposition}  
\newtheorem{Lemma}[Theorem]{Lemma} 
\newtheorem{Corollary}[Theorem]{Corollary}
\newtheorem{Remark}[Theorem]{Remark}
\newtheorem{Example}[Theorem]{Example}
    
\newtheorem{Conjecture}[Theorem]{Conjecture}

\theoremstyle{definition}

\begin{document} 


\title[Critical ideals of signed graphs with twin vertices]{Critical ideals of signed graphs with twin vertices}


\author{Carlos A. Alfaro}
\address{
Banco de M\'exico\\
Calzada Legaria 691, m\'odulo IV\\
Col. Irrigaci\'on\\
11500 Ciudad de M\'exico, M\'exico.}
\email[Carlos A. Alfaro]{alfaromontufar@gmail.com}
\author{Hugo Corrales}
\address{
Escuela Superior de Economia\\
Plan de Agua Prieta No. 66\\
Col. Plutarco El\'ias Calles\\
11340 Ciudad de M\'exico, M\'exico.
} 
\email[Hugo~Corrales]{hhcorrales@gmail.com}
\author{Carlos E. Valencia}
\address{
Departamento de Matem\'aticas\\
Centro de Investigaci\'on y de Estudios Avanzados del IPN\\
Apartado Postal 14--740\\
07000 Ciudad de M\'exico, M\'exico.
} 

\email[Carlos E.~Valencia\footnote{Corresponding author}]{cvalencia@math.cinvestav.edu.mx, cvalencia75@gmail.com}
\thanks{The authors were partially supported by SNI}


\begin{abstract} 
This paper studies critical ideals of graphs with twin vertices, which are vertices with the same neighbors.
A pair of such vertices are called replicated if they are adjacent, and duplicated, otherwise.
Critical ideals of graphs having twin vertices have good properties and show regular patterns. 
Given a graph $G=(V,E)$ and ${\bf d}\in \mathbb{Z}^{|V|}$, let $G^{\bf d}$ be the graph obtained from $G$ by duplicating ${\bf d}_v$ times or replicating $-{\bf d}_v$ times the vertex $v$ when ${\bf d}_v>0$ or ${\bf d}_v<0$, respectively.  
Moreover, given $\delta\in \{0,1,-1\}^{|V|}$, let 
\[
\mathcal{T}_{\delta}(G)=\{G^{\bf d}: {\bf d}\in \mathbb{Z}^{|V|} \text{ such that } {\bf d}_v=0 \text{ if and only if }\delta_v=0 \text{ and } {\bf d}_v\delta_v>0 \text{ otherwise}\}
\]
be the set of graphs sharing the same pattern of duplication or replication of vertices.
More than one half of the critical ideals of a graph in $\mathcal{T}_{\delta}(G)$ can be determined by the critical ideals of $G$.
The algebraic co-rank of a graph $G$ is the maximum integer $i$ such that the $i$-{\it th} critical ideal of $G$ is trivial. 
We show that the algebraic co-rank of any graph in $\mathcal{T}_{\delta}(G)$ is equal to the algebraic co-rank of $G^{\delta}$. 
Moreover, the algebraic co-rank can be determined by a simple evaluation of the critical ideals of $G$.
For a large enough ${\bf d}\in \mathbb{Z}^{V(G)}$, we show that the critical ideals of $G^{\bf d}$ have similar behavior to the critical ideals 
of the disjoint union of $G$ and
some set $\{K_{n_v}\}_{\{v\in V(G) \, |\, {\bf d}_v<0\}}$ of complete graphs and some set $\{T_{n_v}\}_{\{v\in V(G) \, |\, {\bf d}_v>0\}}$ of trivial graphs. 
Additionally, we pose important conjectures on the distribution of the algebraic co-rank of the graphs with twins vertices.
These conjectures imply that twin-free graphs have a large algebraic co-rank, meanwhile a graph having small algebraic co-rank has at least one pair of twin vertices.
\end{abstract}

\maketitle


\section{Introduction and background} 
A signed multidigraph $G_\sigma$ is a pair that consists of a multidigraph $G$ (a digraph possibly with multiple arcs) and a function $\sigma$, 
called the sign, from the arcs of $G$ into the set $\{1,-1\}$.
Along the paper, all digraphs are allowed to have multiple signed arcs; when digraphs have neither multiple nor signed arcs, then we refer them as graphs.
Given a set of variables $X_G=\{x_u\, :\, u\in V(G)\}$ indexed by the vertices of $G$ and a principal ideal domain (PID) $\mathcal{P}$,
the generalized Laplacian matrix $L(G_\sigma,X_G)$ of $G_\sigma$ is the matrix whose entries are given by
\[
L(G_\sigma,X_G)_{uv}=\begin{cases}
x_u& \text{ if } u=v,\\
-\sigma(uv) m_{uv} 1_{\mathcal{P}}& \text{ otherwise},
\end{cases}
\]
where $m_{uv}$ is the number of arcs leaving $u$ and entering to $v$, and $1_{\mathcal{P}}$ is the identity of $\mathcal{P}$.
Moreover, if $\mathcal{P}[X_G]$ is the polynomial ring over $\mathcal{P}$ in the variables $X_G$, then the critical ideals of $G_\sigma$ are the determinantal ideals given by
\[
I_i(G_\sigma,X_G)=\langle \{ {\rm det} (m) \, : \, m \text{ is an }i\times i \text{ submatrix of }L(G_\sigma,X_G)\}\rangle\subseteq \mathcal{P}[X_G],
\]
for all $1\leq i\leq |V(G)|$.
We say that a critical ideal is trivial when it is equal to $\langle1_{\mathcal{P}}\rangle$.
For simplicity. we write $I_i(G_\sigma,X)$ instead of $I_i(G_\sigma,X_G)$.

\begin{Definition}
The algebraic co-rank $\gamma_\mathcal{P}(G_\sigma)$ of $G_\sigma$ is the maximum integer $i$ such that $I_i(G_\sigma,X)$ is trivial.
\end{Definition}

Since $I_n(G_\sigma,X)=\langle {\rm det} (L(G_\sigma,X))\rangle\neq \langle 1\rangle$, $\gamma_\mathcal{P}(G_\sigma)\leq n-1$.
The algebraic co-rank of a graph is closely related to combinatorial properties of the graph.
For instance, if $H_\sigma$ is an induced subgraph of $G_\sigma$, 
then $I_i(H_\sigma,X)\subseteq I_i(G_\sigma,X)$ for all $1\leq i\leq |V(H)|$ (see~\cite[Proposition 3.3]{critical}).
Therefore, $\gamma(H_\sigma)\leq\gamma(G_\sigma)$. 
Also, if $\alpha(G)$ and $\omega(G)$ denote the stability number and the clique number of $G$, respectively, then
\[
\gamma_\mathcal{P}(G)\leq 2(n-\omega(G))+1\text{ and } \gamma_\mathcal{P}(G)\leq 2(n-\alpha(G)),
\]
see \cite[Theorem 3.13]{critical}.

We now introduce the operations of duplication and replication of vertices.
Given a multidigraph $G$ and a vertex $v\in V(G)$, duplicating the vertex $v$ consists in adding a new vertex $v^1$ to $V(G)$ and making it adjacent to each neighbor of $v$, respecting the multiplicities and signs of arcs.
Let $d(G, v)$ denote the multidigraph obtained from $G$ after duplicating the vertex $v$.
Similarly, replicating the vertex $v$ consists in duplicating $v$ and adding the arcs $vv^1$ and $v^1v$.
Let $r(G, v)$ denote the multidigraph obtained from $G$ by replicating the vertex $v$.
Two vertices $u$ and $v$ are called twins if they have the same neighborhood.
In the literature, duplicated vertices are also known as false twins, and replicated vertices are also named true twins.
Let $d^k(G,v)$ denote the multidigraph obtained from $G$ by duplicating the vertex $v$ a total of $k$ times and similarly for $r^k(G,v)$.

Given ${\bf d}\in \mathbb{Z}^{|V|}$, let $G^{\bf d}$ be the graph obtained from $G$ by duplicating the vertex $v$ ${\bf d}_v$ times if ${\bf d}_v>0$, 
and replicating $v$ $-{\bf d}_v$ times if ${\bf d}_v<0$, for each $v\in V(G)$. 
Note that $G=G^{\bf 0}$.
Let $V(G^{\bf d},v)$ denote the vertex set $\{v,v^1,\ldots,v^{|{\bf d}_v|}\}$ created by either duplicating or replicating the vertex $v$.
To simplify the notation, the vertex $v$ will be also denoted by $v^0$.
The following example illustrates these concepts.
\begin{Example}
Let $C_4$ be the cycle with four vertices and ${\bf d}=(-1,1,1,1)$.
Thus $C_4^{\bf d}$ is the graph with eight vertices shown in Figure~\ref{fig:01}.b.
\begin{figure}[h]
\begin{center}
\begin{tabular}{c@{\extracolsep{20mm}}c}
\begin{tikzpicture}[line width=1pt, scale=0.9]
	\tikzstyle{every node}=[inner sep=0pt, minimum width=4.5pt] 
	\draw (225:0.8) node (v1) [draw, circle, fill=gray] {};
	\draw (315:0.8) node (v2) [draw, circle, fill=gray] {};
	\draw (45:0.8) node (v3) [draw, circle, fill=gray] {};
	\draw (135:0.8) node (v4) [draw, circle, fill=gray] {};
	\draw (225:1.1) node () {\small $a$};
	\draw (315:1.1) node () {\small $b$};
	\draw (45:1.1) node () {\small $c$};
	\draw (135:1.1) node () {\small $d$};
	\draw (v1) edge (v2);
	\draw (v2) edge (v3);
	\draw (v3) edge (v4);
	\draw (v4) -- (v1); 
\end{tikzpicture}
&
\begin{tikzpicture}[line width=1pt, scale=0.9]
	\tikzstyle{every node}=[inner sep=0pt, minimum width=4pt] 
	\draw (225:0.8) node (v1) [draw, circle, fill=gray] {};
	\draw (315:0.8) node (v2) [draw, circle, fill=gray] {};
	\draw (45:0.8) node (v3) [draw, circle, fill=gray] {};
	\draw (135:0.8) node (v4) [draw, circle, fill=gray] {};
	\draw (225:1.2) node (v1p) [draw, circle, fill=gray] {};
	\draw (315:1.2) node (v2p) [draw, circle, fill=gray] {};
	\draw (45:1.2) node (v3p) [draw, circle, fill=gray] {};
	\draw (135:1.2) node (v4p) [draw, circle, fill=gray] {};
	\draw (225:1.9) node () {\small $V(C_4^{\bf d}, a)$};
	\draw (315:1.9) node () {\small $V(C_4^{\bf d}, b)$};
	\draw (45:1.9) node () {\small $V(C_4^{\bf d}, c)$};
	\draw (135:1.9) node () {\small $V(C_4^{\bf d}, d)$};
	\draw (225:0.55) node () {\small $a$};
	\draw (315:0.55) node () {\small $b$};
	\draw (45:0.55) node () {\small $c$};
	\draw (135:0.55) node () {\small $d$};
	\draw (215:1.35) node () {\small $a^1$};
	\draw (325:1.4) node () {\small $b^1$};
	\draw (35:1.4) node () {\small $c^1$};
	\draw (145:1.35) node () {\small $d^1$};
	\draw (v1) edge (v2);
	\draw (v2) edge (v3);
	\draw (v3) edge (v4);
	\draw (v4) -- (v1);
	\draw (v1p) edge (v2p);
	\draw (v1p) edge (v2);
	\draw (v1) edge (v2p);
	\draw (v2p) edge (v3p);
	\draw (v2p) edge (v3);
	\draw (v2) edge (v3p);
	\draw (v3p) edge (v4p);
	\draw (v3p) edge (v4);
	\draw (v3) edge (v4p);
	\draw (v4p) -- (v1p);
	\draw (v4p) -- (v1);
	\draw (v4) -- (v1p); 
	\draw (v1p) edge (v1);
\end{tikzpicture}
\\
$(a)$
&
$(b)$
\end{tabular}		
\end{center}
\caption{The cycle with four vertices and $C_4^{(-1,1,1,1)}$.}
\label{fig:01}
\end{figure}
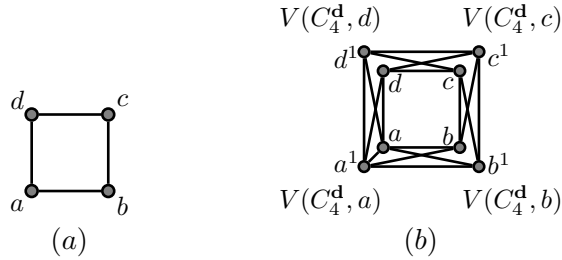
\end{Example}

Critical ideals were defined in \cite{critical} as a refinement of the critical group of a graph.
We now introduce the critical group of a multidigraph. 
The Laplacian matrix $L(G_\sigma)$ of $G_\sigma$ is the evaluation of $L(G_\sigma,X)$ at $X=D_G$, where $D_G$ is the out-degree vector of $G$.
By considering $L(G_\sigma)$ as a linear map $L(G_\sigma):\mathbb{Z}^V\rightarrow \mathbb{Z}^V$, the cokernel of $L(G_\sigma)$ 
is the quotient module $\mathbb{Z}^{V}/{\rm Im}\, L(G_\sigma)$. 
The torsion part of this module is the critical group $K(G_\sigma)$ of $G_\sigma$.
The critical group has been studied intensively on several contexts over the last 30 years, such that: the {\it group of components} \cite{lorenzini1991,lorenzini2008}, 
the {\it Picard group} \cite{bhn,biggs1999}, the {\it Jacobian group} \cite{bhn,biggs1999}, the {\it sandpile group} \cite{cone,cori},  {\it chip-firing game} \cite{biggs1999,merino}, 
or {\it Laplacian unimodular equivalence} \cite{gmw,merris}.
Recently, the critical ideals have played an important role in understanding and classifying the graphs whose 
critical group has $i$ invariant factors equal to one, see~\cite{g2,g3}.
In general, the relations between the critical group and other parameters of a graph G remain unknown.  

There are few natural constructions of graphs which behave well with respect to the critical group.
For example, the critical group $K(G\sqcup H)$ of a disjoint union $G\sqcup H$ of two graphs $G$ and $H$ is isomorphic to $K(G)\oplus K(H)$.
Moreover, in \cite{wagner} it was proved that if the graphic matroids of $G$ and $H$ are isomorphic, then their critical groups are isomorphic.
This was proved by studying the operations of {\it splittings} or {\it mergings of one-vertex cuts} and {\it twistings of two-vertex cuts}.
Other operations on graphs have been explored, such as the {\it cone of a graph} \cite{cone}, the {\it line graph} \cite{bmmpr, levine}, and the {\it clique-inserted graph} \cite{cz}.

The main goal of this article is to give a description of the critical ideals of signed multidigraphs with twin vertices.
More precisely, given a graph $G$ and $\delta\in \{0,1,-1\}^{|V|}$, let 
\[
\mathcal{T}_{\delta}(G)=\{G^{\bf d}: {\bf d}\in \mathbb{Z}^{|V|} \text{ such that } {\rm supp}({\bf d})=\delta\},
\]
where 
\[
{\rm supp}({\bf d})_v=
\begin{cases}
	-1 & \text{ if } {\bf d}_v < 0,\\
	1 & \text{ if } {\bf d}_v > 0,\\
	0 & \text{ otherwise.}
\end{cases}
\]
We prove that more than one half of the critical ideals of the graphs in $\mathcal{T}_{\delta}(G)$ 
are determined by the critical ideals of $G$, see Theorems~\ref{teo:deq} and \ref{teo:req}.
Moreover, the algebraic co-rank of any graph in $\mathcal{T}_{\delta}(G)$ is equal to the algebraic co-rank of $G^{\delta}$ (see Corollary~\ref{coro:bound}),
which is less than or equal to  the number of vertices of $G$ and is determined by a simple evaluation of the critical ideals of $G$.

We illustrate these results by presenting a simple example.
Consider the path $P_3$ with three vertices.
Then $\gamma_{\mathcal{P}}(P_3)=2$ and 
\[
I_3(P_3, X)=\langle x_1x_2x_3-x_1-x_3\rangle=\langle p\rangle.
\]
Our goal is to describe the critical ideals of the graphs obtained by duplicating or replicating some of the vertices of $P_3$
and in particular we are interested in its algebraic co-rank.
For our example we want to calculate the algebraic co-rank of the graphs in one of the following families:
\[
\mathcal{T}_{(-1,-1,-1)}(P_3), \mathcal{T}_{(-1,-1,1)}(P_3), \mathcal{T}_{(-1,1,1)}(P_3), \text{ and }\mathcal{T}_{(-1,1,-1)}(P_3).
\]
Since any graph in one of these families contains $P_3$ as an induced subgraph, its algebraic co-rank is greater than or equal to two.
Theorem~\ref{teo:rd} and Corollary~\ref{coro:bound} imply that the algebraic co-rank of any of these graphs is less than or equal to  three, the number of vertices of $P_3$.
Moreover, all the graphs in each of the families have the same algebraic co-rank and this can be computed by evaluating its third critical ideal.
For instance, the algebraic co-rank of any graph in $\mathcal{T}_{(-1,-1,-1)}(P_3)$ is equal to three because 
\[
p(-1,-1,-1)=(-1)(-1)(-1)-(-1)-(-1)=-1+1+1=1.
\]
A similar argument applies to $\mathcal{T}_{(-1,-1,1)}(P_3)$ and $\mathcal{T}_{(-1,1,1)}(P_3)$.
The case of $\mathcal{T}_{(-1,1,-1)}(P_3)$ is more interesting. 
Since $p(-1,1,-1)=3$, the algebraic co-rank depends on the base ring $\mathcal{P}$.
For instance, if $\mathcal{P}=\mathbb{Z}$, then the algebraic co-rank of any graph in $\mathcal{T}_{(-1,1,-1)}(P_3)$ is two.
However, if $\mathcal{P}$ is a finite field of characteristic different to three, then the algebraic co-rank of any graph in $\mathcal{T}_{(-1,1,-1)}(P_3)$ is three. 

Obtaining the description of the critical ideals of the graphs in a family $\mathcal{T}_{\delta}(G)$ is a difficult task.
However, we can obtain information of more than one half of the critical ideals of the graphs in $\mathcal{T}_{\delta}(G)$, see Remark~\ref{half}.
In Section~\ref{Sbipartite}, the reader will find a description of some of the critical ideals of $\mathcal{T}_{(1,1)}(P_2)$ computed by using results contained in this article.
More precisely, while the vertices are duplicated or replicated several times, the initial critical ideals behave similarly to the critical ideals of the disjoint union of complete and trivial graphs.

These results are important in the study of critical ideals of graphs, in particular, in the understanding of the algebraic co-rank of a graph.
For instance, in the classification of the graphs with algebraic co-rank less than or equal to $k$,
these results allow us to get some insights of the minimal $k$-forbidden graphs, which help in defining the $k$-basic signed graphs. 
It is important to note that there are several important families of graphs in $\bigcup_{(G, \delta)\in \mathcal{G}} \mathcal{T}_{\delta}(G)$ for some set $\mathcal{G}$ of pairs $(G, \delta)$.
For instance, the complete multipartite graphs are equal to $\bigcup_{i=1}^{\infty} \mathcal{T}_{{\bf 1}_i}(K_i)$, 
where $K_i$ is the complete graph with $i$ vertices and ${{\bf 1}_i}$ is the vector of size $i$ where all their entries equal to $1$.
Threshold graphs and quasi-threshold graphs can be described in a similar way.
Moreover, cographs and distance-hereditary graphs are families of graphs with twin vertices. 

The article is structured as follows.
In Section~\ref{rd}, we obtain relations between evaluations of the critical ideals of a signed multidigraph $G$ 
and the critical ideals of the graphs obtained by duplicating or replicating a number of vertices.
Then, we get a partial description of the critical ideals of the graph $G^{\bf d}$ for some ${\bf d}\in \mathbb{Z}^{V(G)}$.
As a consequence, we get an upper bound for the algebraic co-rank of graphs with twins.
To finish this section, we pose conjectures which lead into a wide and interesting outlook of the algebraic co-rank of graphs. 
In Section~\ref{sec:description}, we give precise descriptions of the critical ideals of the $k$-{\it th} duplication and $k$-{\it th} 
replication of vertex $v$ in terms of the critical ideals of $G$.
Finally we present some applications of our results.


\section{An upper bound for the algebraic co-rank of graphs with twins}\label{rd}
The objective of this section is to study critical ideals of graphs with twin vertices.
We begin this section by calculating the minors (which are almost always equal to zero) of the union of matrices in Lemma~\ref{lema:det1}.
By using this lemma, we get a first description for the critical ideals of the graph 
obtained by duplicating or replicating vertices (see Lemma~\ref{lema:d} and Theorem~\ref{teo:rd}).
Then, we get an upper bound for the algebraic co-rank of a graph with twins (see Corollary~\ref{coro:bound}). 
In fact, this bound is tight since the equality holds for the complete graphs (see Example~\ref{example:completa}).
This upper bound can be used in the classification of the graphs that have algebraic co-rank less than or equal to  an integer $k$, see~\cite{g2} and \cite{g3}.

Let $\mathcal P$ be a commutative ring with identity, and let $M_n (\mathcal{P})$ denote the set of $n\times n$ matrices with entries on $\mathcal{P}$. 
Given two vectors ${\bf a}\in \mathcal{P}^{q_1}$ and ${\bf b}\in \mathcal{P}^{q_2}$ and two matrices $P\in M_{p_1\times p_2}(\mathcal{P})$ and 
$Q\in M_{q_1\times q_2}(\mathcal{P})$ such that $p_1+q_1=p_2+q_2$, the join $J(P,{\bf a};Q,{\bf b})$ is the matrix
\[
\left[\begin{array}{cc}
                   P                        &    {\bf 1}_{p_1}^T{\bf b}\\
{\bf a}^T {\bf 1}_{p_2}   &        Q\\
\end{array}\right]
\in M_{p_1+q_1}(\mathcal{P}).
\]
Note that if $G\boxtimes H$ denotes the join of two graphs $G$ and $H$, then 
\[
L(G\boxtimes H, X)=J(L(G,X), -{\bf 1}; L(H,X), -{\bf 1}).
\]

\begin{figure}[h!]
\begin{tabular}{c@{\extracolsep{2cm}}c@{\extracolsep{2cm}}c}
\begin{tikzpicture}[scale=1, line width=0.7pt]
\tikzstyle{every node}=[minimum width=4.5pt, inner sep=0pt, circle]
\draw (0,1) node (v1) [draw, fill=gray, label=above:{\small $v_1$}] {};
\draw (0,-1) node (v2) [draw, fill=gray, label=below:{\small $v_2$}] {};
\draw (v1) edge (v2);
\end{tikzpicture}
&
\begin{tikzpicture}[scale=1, line width=0.7pt]
\tikzstyle{every node}=[minimum width=4.5pt, inner sep=0pt, circle]
\draw (0,1) node (u1) [draw, fill=gray, label=above:{\small $u_1$}] {};
\draw (0,0) node (u2) [draw, fill=gray, label=right:{\small $u_2$}] {};
\draw (0,-1) node (u3) [draw, fill=gray, label=below:{\small $u_2$}] {};
\draw (u1) -- (u2) -- (u3);
\end{tikzpicture}
&
\begin{tikzpicture}[scale=1, line width=0.7pt]
\tikzstyle{every node}=[minimum width=4.5pt, inner sep=0pt, circle]
\draw (-1,1) node (v1) [draw, fill=gray, label=above:{\small $v_1$}] {};
\draw (-1,-1) node (v2) [draw, fill=gray, label=below:{\small $v_2$}] {};
\draw (v1) edge (v2);
\draw (1,1) node (u1) [draw, fill=gray, label=above:{\small $u_1$}] {};
\draw (1,0) node (u2) [draw, fill=gray, label=right:{\small $u_2$}] {};
\draw (1,-1) node (u3) [draw, fill=gray, label=below:{\small $u_3$}] {};
\draw (u1) -- (u2) -- (u3);
\draw (v1) -- (u1) -- (v2) -- (u2) -- (v1) -- (u3) -- (v2) -- (u3) -- (v1);
\end{tikzpicture}
\\
$P_2$ & $P_3$ & $P_2\boxtimes P_3$
\end{tabular}
\caption{The join of two paths.}
\label{fig:JoinTwoPaths}
\end{figure}
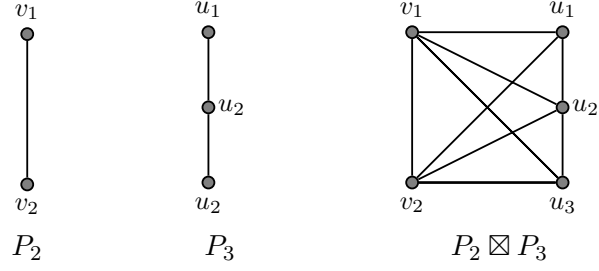
	
\begin{Example}
Consider the join of a path $P_2$ with $2$ vertices and a path $P_3$ with $3$ vertices, see Fig.~\ref{fig:JoinTwoPaths}.
Then,
\begin{eqnarray*}
L(P_2\boxtimes P_3,X_{P_2\boxtimes P_3})&=&J(L(P_2,X_{P_2}),{\bf -1};L(P_3,X_{P_3}),{\bf -1})\\
&=&
\left[
\begin{array}{cc|ccc}
x_{v_1} & -1 & -1 & -1 & -1 \\
-1 & x_{v_2} & -1 & -1 & -1 \\
\hline
-1 & -1 & x_{u_1} &-1 & 0 \\
-1 & -1 & -1 & x_{u_2} &-1 \\
-1 & -1 & 0 & -1 & x_{u_3} \\
\end{array}
\right].
\end{eqnarray*}	
\end{Example}

The following lemma describes the determinant of the join $J(P,{\bf a}; Q,{\bf b})$.

\begin{Lemma}\label{lema:det1}
If $P\in M_{p_1\times p_2}(\mathcal{P})$, $Q\in M_{q_1\times q_2}(\mathcal{P})$ with $p_1+q_1=p_2+q_2$, 
${\bf a}\in \mathcal{P}^{q_1}$, and ${\bf b}\in \mathcal{P}^{q_2}$, then
\[
{\rm det}(J(P,{\bf a}; Q,{\bf b}))=
\begin{cases}
{\rm det}(P)\cdot {\rm det}(Q)-
{\rm det}
\left[\begin{array}{cc}
P&{\bf 1}^T\\
{\bf 1}&0
\end{array}\right] 
\cdot {\rm det}
\left[\begin{array}{cc}
0&{\bf b}\\
{\bf a}^T&Q
\end{array}\right]
& \text{ if } p_1=p_2,\\
\\
{\rm det}
\left[\begin{array}{cc}
P & {\bf 1}^T
\end{array}\right]
\cdot
{\rm det}
\left[\begin{array}{c}
{\bf b}\\
Q
\end{array}\right]
& \text{ if } p_1=p_2+1,\\
\\
{\rm det}
\left[\begin{array}{c}
P\\
{\bf 1}
\end{array}\right]
\cdot
{\rm det}
\left[\begin{array}{cc}
{\bf a}^T &Q
\end{array}\right]
& \text{ if } p_2=p_1+1,\\
0 & \text{ otherwise.}
\end{cases}
\] 
\end{Lemma}
\begin{proof}
The proof follows by induction on $p_1+p_2$.
Note that, if $P\in M_{1\times 0}(\mathcal{P})$, then $[\begin{array}{cc} P & {\bf 1}\end{array}]=[1]$.
Also, if $P\in M_{0\times 1}(\mathcal{P})$, then $[\begin{array}{cc} P & {\bf 1}\end{array}]^T=[1]$.
\end{proof}

Note that all square submatrices of a join of matrices are, in fact, a join of matrices.
Hence, almost all minors of the join of matrices are equal to zero.
This fact will be useful in obtaining a description of the critical ideals of a graph with twin vertices (see Lemmas~\ref{lema:d}, \ref{lema:r}, \ref{lema:gend} and~\ref{lema:genr}).

Given ${\bf a}\in \mathcal{P}^n$, $L\in M_n(\mathcal{P})$ and $1\leq j\leq n$, let ${\rm minors}_j(L, {\bf a})$ denote the set
\[
\left\{\mathrm{det}(M): M\in M_j(\mathcal{P}) \text{ and } M= \left[\begin{array}{c} {\bf a}'\\L'\end{array}\right] \text{ for a submatrix } {\bf a}'\neq \emptyset \text{ of } {\bf a} \text{ and } L'\text{ of } L, \text{ resp.}\right\}.
\]
In a similar way, let ${\rm minors}_j({\bf a}, L)$ be the set of determinants of some submatrices of $\left[\begin{array}{cc}{\bf a}^T& L\end{array}\right]$ of size $j$.
Note that ${\rm minors}_{0}({\bf a}, L)={\rm minors}_{0}(L,{\bf b})=\emptyset$, ${\rm minors}_{1}({\bf a}, L) =\{{\bf a}_i\}_{1\leq i\leq n}$, 
and ${\rm minors}_{1}(L,{\bf b})=\{{\bf b}_i\}_{1\leq i\leq n}$.
Let $M_j(L)$ denote the set of submatrices of $L$ of size $j$.
 
Let $G$ be a signed multidigraph with $n\geq 2$ vertices and $v$ be a vertex of $G$.
It is not difficult to see that
\[
L(G,X)=\left[\begin{array}{cc}x_v& {\bf b} \\ {\bf a}^T & L(G-v,X) \end{array}\right]=J(x_v, {\bf a}; L(G-v,X), {\bf b})
\]
for some ${\bf a}, {\bf b}\in \mathcal{P}^{n-1}$.
The following proposition tell us that the $j$-{\it th} critical ideal of $G$ is generated by four types of minors of $L(G,X)$.
 
\begin{Proposition}\label{eq:g}
If $G$ is a signed multidigraph with $n\geq 2$ vertices and $v$ is a vertex of $G$, then the critical ideal $I_j(G,X)$ of $G$ is equal to
\begin{eqnarray*}
\left\langle{\rm minors}_{j}(L(G-v,X)), {\rm minors}_{j}({\bf a}, L(G-v,X)),{\rm minors}_{j}(L(G-v,X),{\bf b}),\right.\\
\left.\left\{x_v\cdot {\rm det}(M)+{\rm det}( J(0,{\bf a'};M,{\bf b'}))  \, :\, J(x_v,{\bf a'};M,{\bf b'}) \in M_{j}(L(G,X)) \text{ with } {\bf a',b'} \text{ subvectors of } {\bf a,b,}, \text{ resp.}\right\}\right\rangle
\end{eqnarray*}
for all $1\leq j\leq n-1$, and equal to $\left\langle x_v\cdot {\rm det}(L(G-v,X))+{\rm det}(J(0,{\bf a};L(G-v,X),{\bf b}))\right\rangle$ when $j=n$.
\end{Proposition}
\begin{proof}
The proof is simple and is similar to the one given in \cite[Claim 3.12]{critical}.
\end{proof}

We now give a description of the critical ideals of $d(G,v)$ in terms of the critical ideals of $G$.
Let $Y$ be a subset of the variables associated to the vertices of $G$ and ${\bf a}\in\mathcal{P}^{|Y|}$.
Through the paper, $I(G,X)|_{Y={\bf a}}$ will denote the evaluation of $I(G,X)$ at $Y={\bf a}$, and ${\rm minors}_{j}({\bf a}, L,{\bf b})$ will be the set
\[
\left\{ {\rm det}(M)\, : \, M=J(0,{\bf a'};M,{\bf b'}) \in M_j\left(J(0,{\bf a};L,{\bf b})\right)\text{ with } {\bf a',b'}\neq \emptyset \text{ subvectors of } {\bf a,b,} \text{ respectively}\right\}.
\]
Note that ${\rm minors}_{1}({\bf a}, L(G- v, X),{\bf b})=\emptyset$.

\begin{Lemma}\label{lema:d}
Let $G$ be a signed multidigraph with $n\geq 2$ vertices, $v\in V(G)$ and $v^1$ a duplication of $v$.
Then 
\[
I_{j}(d(G,v),X)\subseteq \langle x_{v},x_{v^1}, I_{j}(G,X)|_{x_v=0}\rangle,
\]
for all $1\leq j\leq n$. 
Moreover, $I_{j}(d(G,v),X)$ is trivial if and only if $I_{j}(G,X)|_{x_v=0}$ is trivial.
\end{Lemma}
\begin{proof}
The main idea is to give a description of the $j$-{\it th} critical of $d(G,v)$ in terms of some types of minors, 
similar to the one given in Proposition~\ref{eq:g}, and then use this description to prove the containment.

First, it is not difficult to see that $I_1(d(G, v),X)=\langle x_{v},x_{v^1}, I_1(G,X)|_{x_v=0}\rangle$.
Now, let $\mathcal{I, I'}\subseteq [n+1]$ be two sets of size $j$, and  $\mathcal I_{\{1,2\}}=\{1,2\}\cap \mathcal I$ and $\mathcal I'_{\{1,2\}}=\{1,2\}\cap \mathcal I'$.
Without loss of generality, we might order the vertices such that $x_{v^1}$ is in the entry $(1,1)$ and $x_v$ is in the entry $(2,2)$ of $L(d(G,v),X)$.
Clearly, 
\[
L(d(G,v),X)=J({\rm diag}(x_{v^1}, x_{v}), {\bf a}; L(G- v, X), {\bf b}),
\]
where $L(G,X)=J(x_v, {\bf a}; L(G-v,X), {\bf b})$ for some ${\bf a}, {\bf b}\in \mathcal{P}^{n-1}$.
Let $m_{\mathcal{ I,I'}}={\rm det}(L(d(G,v), X)[\mathcal{I,I'}])\in I_j(d(G,v),X)$.

If $\mathcal I_{\{1,2\}}= \mathcal I'_{\{1,2\}}=\{a\}$, then Lemma~\ref{lema:det1} implies that for some matrix 
$J(x_{v},{\bf a'};M,{\bf b'})\in M_{j}(L(G,X))$ with $M\in M_{j-1}(L(G-v,X))$,
\[
m_{\mathcal{I,I'}}={\rm det}(J(x_{v^a},{\bf a'};M,{\bf b'})) = x_{v^a}\cdot {\rm det}(M)+{\rm det}(J(0,{\bf a'};M,{\bf b'})). 
\]
And, if $|\mathcal I_{\{1,2\}}|, |\mathcal I'_{\{1,2\}}|=1$ and $\mathcal I_{\{1,2\}}\cap \mathcal I'_{\{1,2\}}=\emptyset$, then 
$m_{\mathcal {I,I'}}={\rm det}(J(0,{\bf a'};M,{\bf b'}))$ for some $J(0,{\bf a'};M,{\bf b'}) \in M_{j}(L(G,X))$.
On the other hand, since ${\rm det}(J(x,1;1,0))={\rm det}(J(x,0;1,1))=x$,
\[
m_{\mathcal{I,I'}}\in 
\begin{cases}
\left\{ x_{v^i}\cdot {\rm minors}_{j-1}({\bf a}, L(G- v, X))\right\}_{i=0}^1 & \text{ if } |\mathcal I_{\{1,2\}}|=2, |\mathcal I'_{\{1,2\}}|=1,\\
\left\{ x_{v^i}\cdot {\rm minors}_{j-1}(L(G- v, X), {\bf b})\right\}_{i=0}^1 & \text{ if } |\mathcal I_{\{1,2\}}|=1, |\mathcal I'_{\{1,2\}}|=2.
\end{cases}
\]
Finally, since ${\rm det}(J({\rm diag}(x_{v^1},x_{v}),(1,1);0,(1,1)))=-(x_{v^1}+x_{v})$, Lemma~\ref{lema:det1} implies that $m_{\mathcal {I,I'}}$ belongs to
{\small
\[
S_{j}(G,v)\,=\, \left\{x_{v}x_{v^1}\cdot {\rm det}(M)\,+\,(x_{v}+x_{v^1})\cdot {\rm det}(J(0,{\bf a'};M,{\bf b'})) \, :\, 
J(x_v,{\bf a}';M,{\bf b}') \in M_{j-1}(L(G,X)) \text{ with } {\bf a}',{\bf b}'\neq \emptyset \right\},
\] 
}
when $\mathcal I_{\{1,2\}}$ and $\mathcal I'_{\{1,2\}}$ are equal to $\{1,2\}$.
By convention $S_{1}(G,v)=\{x_v\}$ and $S_{2}(G,v)=\{x_{v}x_{v^1}\}$.

Therefore, for $1\leq j\leq n-1$, the $j$-{\it th} critical ideal of $d(G,v)$ has the following expression:
\begin{eqnarray}\label{eq:d}
\nonumber 
I_{j}(d(G,v),X)&=&\langle {\rm minors}_{j}(L(G- v, X)), \left\{ x_{v^i}\cdot {\rm minors}_{j-1}(L(G- v, X))\right\}_{i=0}^1,\\
\nonumber &&
{\rm minors}_{j}({\bf a}, L(G- v, X)),\left\{ x_{v^i}\cdot {\rm minors}_{j-1}({\bf a}, L(G- v, X))\right\}_{i=0}^1,\\
&&
{\rm minors}_{j}(L(G- v, X),{\bf b}), \left\{ x_{v^i}\cdot {\rm minors}_{j-1}(L(G- v, X), {\bf b})\right\}_{i=0}^1,\\
\nonumber &&
{\rm minors}_{j}({\bf a}, L(G- v, X),{\bf b}), S_{j}(G,v)
\rangle.
\end{eqnarray}
Thus, $I_2(d(G, v),X)\subseteq \langle x_{v},x_{v^1}, I_2(G,X)|_{x_v=0}\rangle$. 
Also, in a similar way, $I_{n}(d(G,v),X)$ is equal to 
\begin{eqnarray*}
\langle  \left\{ x_{v^i}\cdot {\rm det}(L(G-v,X))\right\}_{i=0}^1, 
{\rm det}( J(0,{\bf a};L(G-v,X),{\bf b})),
\left\{ x_{v^i}\cdot {\rm minors}_{n-1}({\bf a}, L(G-v,X))\right\}_{i=0}^1,\\
\left\{ x_{v^i}\cdot {\rm minors}_{n-1}(L(G-v,X), {\bf b})\right\}_{i=0}^1,
S_{n}(G,v) \rangle. 
\end{eqnarray*}

On the other hand, by Proposition~\ref{eq:g} we have that $I_{j}(G, X)|_{x_v=0}$ is equal to 
{
\[
\langle  {\rm minors}_{j}(L(G-v,X)), 
{\rm minors}_{j}({\bf a}, L(G-v,X)), {\rm minors}_{j}(L(G-v,X), {\bf b}), {\rm minors}_{j}({\bf a},L(G-v,X),{\bf b})
\rangle,
\]
}
\noindent for $1\leq j\leq n-1$, and $I_{n}(G, X)|_{x_v=0}=\langle {\rm det }(L(G,X)|_{x_v=0})\rangle=
\langle {\rm det} (J(0,{\bf a}; L(G-v,X_{G-v}),{\bf b})) \rangle$.
By using the previous equalities, we get that
\[
I_{j}(d(G,v),X) \subseteq \langle x_{v}, x_{v^1}, I_{j}(G, X)|_{x_v=0} \rangle
\] 
for $1\leq j\leq n$.
Therefore $I_{j}(d(G,v),X)$ is trivial if and only if $I_{j}(G,X)|_{x_v=0}$ is trivial.
\end{proof}

Now, we give a description of the critical ideals of the replication of a vertex of a signed multidigraph.

\begin{Lemma}\label{lema:r}
Let $G$ be a signed multidigraph with $n\geq 2$ vertices and $v$ be a vertex of $G$.
Then 
\[
I_{j}(r(G,v),X)\subseteq \langle x_{v}+1,x_{v^1}+1, I_{j}(G,X)|_{x_v=-1}\rangle,
\]
for all $1\leq j \leq n$.
Moreover, $I_{j}(r(G,v),X)$ is trivial if and only if $I_{j}(G,X)|_{x_v=-1}$ is trivial. 
\end{Lemma}
\begin{proof}
We will give an analogous proof to the one of Lemma~\ref{lema:d}.
We need to make a significant difference in the identity ${\rm det}(J(-1,{\bf a};M,{\bf b})) = -{\rm det}(M)+{\rm det}(J(0,{\bf a};M,{\bf b}))$.
Firstly, for all $1\leq j\leq n-1$, the $j$-{\it th} critical ideal of the graph obtained by replicating vertex $v$ in $G$ has the following expression:
\begin{eqnarray}\label{eq:r}
\nonumber I_{j}(r(G,v),X)&=&\langle {\rm minors}_{j}(L(G-  v),X), \left\{ (x_{v^i}+1)\cdot {\rm minors}_{j-1}(L(G-v,X))\right\}_{i=0}^1,\\
\nonumber && 
{\rm minors}_{j}({\bf a}, L(G-v,X)),\left\{ (x_{v^i}+1)\cdot {\rm minors}_{j-1}({\bf a},L(G-v,X))\right\}_{i=0}^1,\\
&&
{\rm minors}_{j}(L(G-v,X), {\bf b}), \left\{ (x_{v^i}+1)\cdot {\rm minors}_{j-1}(L(G-v,X), {\bf b})\right\}_{i=0}^1,\\
\nonumber && R_j(G,v), \widetilde{S}_{j}(G,v) \rangle,
\end{eqnarray}
\noindent where $R_j(G,v)=\left\{ {\rm det}(J(-1,{\bf a'};M,{\bf b'})) = -{\rm det}(M)+{\rm det}(J(0,{\bf a'};M,{\bf b'})) \, :\, J(x_v,{\bf a'};M,{\bf b'}) \in M_{j}(L(G,X))\right\}$
and
{
$\widetilde{S}_{j}(G,v)=\{(x_{v}+1)(x_{v^1}+1)\cdot {\rm det}(M)+((x_{v}+1)+(x_{v^1}+1))\cdot 
{\rm det}(J(-1,{\bf a'};M,{\bf b'})) \, :\, J(x_v,{\bf a'};M,{\bf b'}) \in M_{j-1}(L(G,X))\}$.
}
Besides, the $n$-{\it th} critical ideal of $r(G,v)$ has the following expression:
\begin{eqnarray*}
I_{n}(r(G,v),X)&=& \langle  \left\{ (x_{v^i}+1)\cdot {\rm det}(L(G-v,X))\right\}_{i=0}^1, 
{\rm det}(J(-1,{\bf a};L(G-v,X),{\bf b}) ),\\ 
& & \left\{ (x_{v^i}+1)\cdot {\rm minors}_{n-1}({\bf a},L(G-v,X))\right\}_{i=0}^1, \\
& & \left\{ (x_{v^i}+1)\cdot {\rm minors}_{n-1}(L(G-v,X), {\bf b})\right\}_{i=0}^1,
\widetilde{S}_{n}(G,v)\rangle. 
\end{eqnarray*}

On the other hand,  by Proposition \ref{eq:g} we have that
$I_{j}(G, X)|_{x_v=-1}$ is equal to 
\[
\langle  {\rm minors}_{j}(L(G-v,X)), {\rm minors}_{j}({\bf a}, L(G-v,X)), {\rm minors}_{j}(L(G-v,X), {\bf b}), R_j(G,v) \rangle,
\]
for all $1\leq j\leq n-1$, and $I_{n}(G, X)|_{x_v=-1}=\langle {\rm det }(L(G,X))|_{x_v=-1}\rangle=
\langle {\rm det}(J(-1,{\bf a};L(G-v,X),{\bf b}) ) \rangle$.
Therefore, 
\[
I_{j}(r(G,v),X) \subseteq \langle x_{v}+1, x_{v^1}+1, I_{j}(G, X)|_{x_v=-1} \rangle
\]
for all $1\leq j\leq n$.
Finally, it is clear that $I_{j}(r(G,v),X)$ is trivial if and only if $I_{j}(G,X)_{x_v=-1}$ is trivial.
\end{proof}

The next example shows a signed multidigraph satisfying the equality in the inclusions given in Lemmas~\ref{lema:d} and~\ref{lema:r}.
\begin{Example}
Let $G$ be the cycle with five vertices, where the arcs $v_2v_1$ and $v_1v_5$ have negative signs, see Figure~\ref{fig:00}.

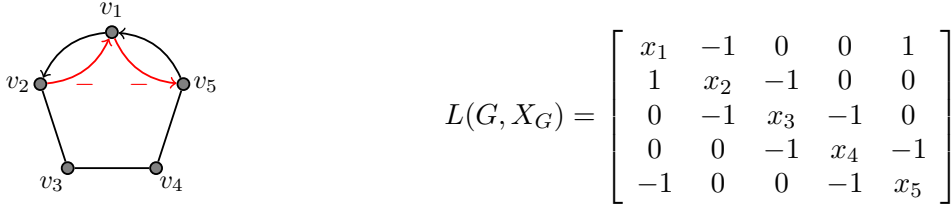
\begin{figure}[h]
\begin{center}
\begin{tabular}{c@{\extracolsep{2cm}}c}
\multirow{9}{40mm}{
\vspace{40mm}
\begin{tikzpicture}[scale=1, line width=0.7pt]
\tikzstyle{every node}=[minimum width=4.5pt, inner sep=0pt, circle]
\draw (72+18:1) node (v1) [draw, fill=gray, label=above:{\small $v_1$}] {};
\draw (144+18:1) node (v2) [draw, fill=gray, label=left:{\small $v_2$}] {};
\draw (216+18:1) node (v3) [draw, fill=gray, label=below left:{\small $v_3$}] {};
\draw (288+18:1) node (v4) [draw, fill=gray, label=below right:{\small $v_4$}] {};
\draw (18:1) node (v5) [draw, fill=gray, label=right:{\small $v_5$}] {};
\draw(v2) edge (v3);
\draw (v3) edge (v4);
\draw (v4) edge (v5);
\path[->,bend right, red] (v2) edge node[below] {\small $-$} (v1)
(v1) edge node[below] {\small $-$} (v5);
\path[->,bend right] (v5) edge (v1)
(v1) edge (v2);
\end{tikzpicture}
}
&
\\
&
$
L(G, X_G)=
\left[\begin{array}{cccccc}
x_1 & -1 & 0 &  0 & 1 \\
 1 & x_2 & -1 & 0 & 0 \\
 0 & -1 &  x_3 & -1 & 0 \\
 0 &   0 &  -1 &  x_4 & -1 \\
 -1 &   0 &  0 &  -1 & x_5 \\
\end{array}\right]
$
\end{tabular}
\end{center}
\caption{A signed multidigraph $G$ with five vertices and its generalized Laplacian matrix.}
\label{fig:00}
\end{figure}
It can be check that the algebraic co-rank of the graph $G$ is equal to 3, when $\mathcal P = \mathbb Z$.
Since $I_4(G,X)$ is given by $\langle x_1x_2+x_4+1, x_2x_3-x_5-1, x_3x_4+x_1-1, x_4x_5-x_2-1, x_1x_5+x_3+1 \rangle$,
$I_4(G,X)|_{x_{1}=0} = \langle x_3+1, x_4+1, x_3x_4-1, x_2x_3-x_5-1,x_4x_5-x_2-1 \rangle$ = $ \langle x_3+1, x_4+1, x_2+x_5+1 \rangle$
and 
\begin{eqnarray*}
I_4(G,X)|_{x_{1}=-1}&=& \langle -x_5+x_3+1, -x_2+x_4+1, x_4x_5-x_2-1, x_3x_4-2, x_2x_3-x_5-1 \rangle\\
&=&\langle x_3-x_5+1, x_2-x_4-1, x_4x_5-x_4-2 \rangle.
\end{eqnarray*}
On the other hand, the $4$-{\it th} critical ideal $I_4(d(G,v_1),X)$ is equal to $\langle x_{1}, x_{1^1}, x_3+1, x_4+1, x_2+x_5+1 \rangle$, and the $4$-{\it th} critical ideal  $I_4(r(G,v_1),X)$ is equal to 
\[
\langle x_{1}+1, x_{1^1}+1, x_3-x_5+1, x_2-x_4-1, x_4x_5-x_4-2 \rangle.
\]
\end{Example}

Successive applications of Lemmas~\ref{lema:d} and~\ref{lema:r} lead to the following result:

\begin{Theorem}\label{teo:rd}
Let $G$ be a signed multidigraph with $n\geq 2$ vertices, ${\bf d}\in \mathbb{Z}^{n}$, and
\[
\phi({\bf d})_v =
\begin{cases}
0 & \text{ if }{\bf d}_v>0,\\
-1 & \text{ if }{\bf d}_v<0,\\
x_v & \text{ if }{\bf d}_v=0.
\end{cases}
\]
Then the $j$-{\it th} critical ideal $I_{j}(G^{\bf d},X)$ is included in the ideal
\[
\left\langle \{\{x_{v^i}\}_{i=0}^{{\bf d}_v}\, : {\bf d}_v\geq 1\},\{\{x_{v^i}+1\}_{i=0}^{-{\bf d}_v}\, : \, {\bf d}_v\leq -1\}, I_{j}(G,X)|_{X=\phi({\bf d})} \} \right\rangle \text{ for all }1\leq j \leq n.
\]
Moreover, $I_{j}(G^{\bf d},X)$ is trivial if and only if $I_{j}(G,X)|_{X=\phi({\bf d})}$ is trivial.
\end{Theorem}

This theorem shows that the algebraic co-rank of $G^{\bf d}$ is determined by an evaluation of the critical ideals of $G$.
It is well known \cite{critical} that the evaluation of the critical ideals of $G$ determines the critical group of $G$.
These facts open the question about the meaning of another evaluations of the critical ideals of a graph.

Next example illustrates Lemma~\ref{lema:d}, Lemma~\ref{lema:r} and Theorem~\ref{teo:rd}.
\begin{Example}
Let $G$ be the graph given by Figure~\ref{fig:0}.
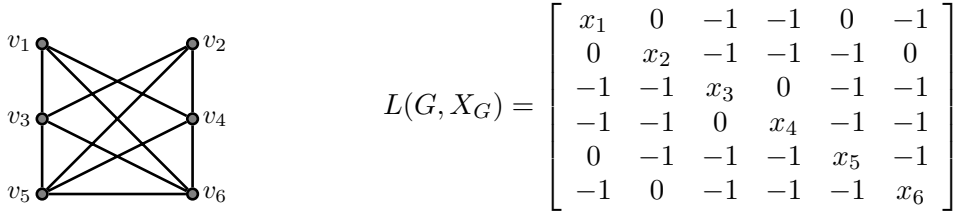
\begin{figure}[h]
\begin{center}
\begin{tabular}{c@{\extracolsep{2cm}}c}
\multirow{9}{3cm}{
	\begin{tikzpicture}[line width=1pt, scale=1]
		\tikzstyle{every node}=[inner sep=0pt, minimum width=4pt] 
		\draw (-1,1) node (v1) [draw, circle, fill=gray] {};
		\draw (1,1) node (v2) [draw, circle, fill=gray] {};
		\draw (-1,0) node (v3) [draw, circle, fill=gray] {};
		\draw (1,0) node (v4) [draw, circle, fill=gray] {};
		\draw (-1,-1) node (v5) [draw, circle, fill=gray] {};
		\draw (1,-1) node (v6) [draw, circle, fill=gray] {};
		\draw (v1)+(-0.3,0) node () {\small $v_1$};
		\draw (v2)+(0.3,0) node () {\small $v_2$};
		\draw (v3)+(-0.3,0) node () {\small $v_3$};
		\draw (v4)+(0.3,0) node () {\small $v_4$};
		\draw (v5)+(-0.3,0) node () {\small $v_5$};
		\draw (v6)+(0.3,0) node () {\small $v_6$};
		\draw (v1) -- (v3) -- (v5) -- (v6) -- (v3) -- (v2) -- (v4) -- (v6) -- (v1) -- (v4) -- (v5) -- (v2);
	\end{tikzpicture}
}
& \\
&
$
L(G, X_G)=
\left[\begin{array}{cccccc}
x_1 &    0 &     -1 &    -1 &     0 &   -1 \\
   0 & x_2 &     -1 &    -1 &    -1 &    0 \\
  -1 &    -1 &  x_3 &     0 &    -1 &   -1 \\
  -1 &    -1 &      0 &  x_4 &   -1 &   -1 \\
   0 &    -1 &     -1 &    -1 & x_5 &   -1 \\
  -1 &     0 &     -1 &    -1 &   -1 & x_6 
\end{array}\right]
$\\
& \\
\end{tabular}
\end{center}
\caption{A graph $G$ with eight vertices and its generalized Laplacian matrix.}
\label{fig:0}
\end{figure}

Using a computer algebra system, we can see that $\gamma_{\mathbb{Z}}(G)=3$ and their non-trivial critical ideals are the following:
{\small
\begin{eqnarray*}
I_{4}(G,X) &\!\!\!\!\!\!=\!\!\!\!\!\!& \langle  x_3, x_4, x_1x_2+1, (x_1-1)x_6-2, (x_2-1)x_5-2, x_1x_5+x_5+2x_1, x_2x_6+x_6+2x_2, x_5x_6+x_6+x_5+2 \rangle,\\
I_5(G,X) &\!\!\!\!\!\!=\!\!\!\!\!\!& \langle x_2x_4x_5(x_6+1)-x_4x_6, x_2x_3x_4+x_2x_3x_6+x_2x_4x_6+2x_2x_3+2x_2x_4+x_3x_6+x_4x_6,\\
&&
x_1x_3x_4+x_1x_3x_5+x_1x_4x_5+2x_1x_3+2x_1x_4+x_3x_5+x_4x_5,  x_1x_4x_6(x_5+1)-x_4x_5, \\ 
&& 
x_1x_4(x_2x_6+x_2+x_6)-x_4, x_1x_3(x_2x_6+x_2+x_6)-x_3, (x_3+x_4)(x_5x_6+x_5+x_6+2)+x_3x_4,\\ 
&&
x_1x_2 (x_6+x_5)+x_5 x_6(x_1+x_2)+2(x_1x_2+x_1x_6+x_2x_5)- x_5-x_6-2 \rangle,\\
I_6(G, X)&\!\!\!\!\!\!=\!\!\!\!\!\!& \langle {\rm det}(L(G,X))\rangle.
\end{eqnarray*}
}
From these equalities and applying Theorem~\ref{teo:rd}, we can easily obtain that the critical ideals 
$I_4(d(G,v_i),X)$ and $I_4(r(G,v_j),X)$ are trivial for all $i\in\{1,2\}$ and $j\in \{3,4\}$.
Furthermore, the ideals $I_4(G^{{\bf e}_1-{\bf e}_6},X)$, $I_4(G^{{\bf e}_1-{\bf e}_5},X)$, $I_4(G^{{\bf e}_2-{\bf e}_5},X)$, $
I_4(G^{{\bf e}_2-{\bf e}_6},X)$, $I_4(G^{{\bf e}_5-{\bf e}_6},X)$, $I_4(G^{{\bf e}_6-{\bf e}_5},X)$ are also trivial. 
On the other hand, 
\[
I_{4}(d(G,v_6),X) = \langle x_6,x_{6^1}, I_4(G,X)|_{x_6=0}\rangle=\langle x_6,x_{6^1}, 2,x_3, x_4, x_5, x_1x_2+1 \rangle,
\] 
\begin{eqnarray*}
I_5(d(G,v_6),X) &\!\!\!\!\!\!=\!\!\!\!\!\!& \langle x_3x_6,x_3x_{6^1},x_4x_6,x_4 x_{6^1}, x_3x_5,x_4x_5,x_6(x_1x_2+1),x_{6^1}(x_1x_2+1),x_6(x_2x_5-x_5-2),\\
&& 
x_{6^1}(x_2x_5-x_5-2), x_6(x_1x_5+x_5+2x_1), x_{6^1}(x_1x_5+x_5+2x_1), x_6x_{6^1}(x_1-1)-2(x_6+x_{6^1}),\\
&& 
x_6x_{6^1}(x_2+1)+2x_2(x_6+x_{6^1}), (x_6x_{6^1}+x_6+x_{6^1})(x_5+1)+(x_6+x_{6^1}), x_3x_4+2x_3+2x_4,\\
&&
x_3(x_1x_2-1), x_4(x_1x_2-1), x_1x_2x_5+2x_1x_2+2x_2x_5-x_5-2\rangle\\
&\!\!\!\!\!\!\subsetneq \!\!\!\!\!\!&\langle x_6,x_{6^1}, I_5(G,X_G)|_{x_6=0}\rangle, \text{ and}
\end{eqnarray*}
\begin{eqnarray*}
I_{5}(G^{{\bf e}_6-{\bf e}_5},X) &\!\!\!\!\!\!=\!\!\!\!\!\!& \langle 2(x_{5^1}\!+\!1), 2(x_5\!+\!1), x_5x_{5^1}\!-\!1, x_6\!+\!x_{6^1},x_6(x_1\!-\!1), 
x_6(x_2\!+\!1), x_6(x_5\!+\!1), \\
&& x_6(x_{5^1}\!+\!1),x_3, x_4,x_1x_2\!-\!2x_2\!-\!1 \rangle\\
&\!\!\!\!\!\!\subsetneq\!\!\!\!\!\!&\langle x_5+1,x_{5^1}+1,x_6,x_{6^1}, I_5(G,X)|_{\{x_6=0, x_5=-1\}}\rangle.
\end{eqnarray*}
Note that $I_5(G,X)|_{\{x_6=0, x_5=-1\}}=\langle x_3, x_4,x_1x_2-2x_2-1 \rangle$, and $x_5x_{5^1}-1=(x_5+1)(x_{5^1}+1)-(x_5+1)-(x_{5^1}+1)$.
\end{Example}

As a consequence of Theorem~\ref{teo:rd}, we get the following bound for the algebraic co-rank of a signed multidigraph with twins.
\begin{Corollary}\label{coro:bound}
If $G$ is a signed multidigraph with $n$ vertices, then 
$\gamma_{\mathcal{P}}(G^{\bf d})=\gamma_{\mathcal{P}}(G^{{\rm supp}({\bf d})})\leq n$ for all ${\bf d}\in \mathbb{Z}^{n}$, where 
\[
{\rm supp}({\bf d})_v=
\begin{cases}
	-1 & \text{ if } {\bf d}_v < 0,\\
	1 & \text{ if } {\bf d}_v > 0,\\
	0 & \text{ otherwise.}
\end{cases}
\]
\end{Corollary}
\begin{proof}
Let ${\bf d}\in \mathbb{Z}^{n}$, $\delta={\rm supp}({\bf d})$, and $\gamma=\gamma_{\mathcal{P}}(G^{\delta})$.
That is, $I_{\gamma}(G^\delta,X)=\langle 1_{\mathcal{P}}\rangle$ and $I_{\gamma+1}(G^\delta,X)\neq \langle 1_{\mathcal{P}}\rangle$.
Since $G^\delta$ is an induced subdigraph of $G^{\bf d}$, by~\cite[Proposition 3.3]{critical} $\gamma_{\mathcal{P}}(G^{\bf d})\geq \gamma$.
Now, we need to prove that $\gamma_{\mathcal{P}}(G^{\bf d})\leq \gamma$, that is, we need to prove that $I_{\gamma+1}(G^{\bf d},X)\neq \langle 1_{\mathcal{P}}\rangle$.

Since $I_{\gamma+1}(G^{\delta}, X)$ is non-trivial and $\phi({\delta})=\phi({\bf d})$, applying Theorem~\ref{teo:rd} to $G$ and $\delta$, 
\[
I_{\gamma+1}(G,X)|_{X=\phi({\delta})}=I_{\gamma+1}(G,X)|_{X=\phi({\bf d})} \neq \langle 1_{\mathcal{P}}\rangle.
\]
Therefore, applying Theorem~\ref{teo:rd} to $G$ and ${\bf d}$ we get that $I_{\gamma+1}(G^{\bf d},X) \neq \langle 1_{\mathcal{P}}\rangle$.

Finally, since $I_{n+1}(G,X)=\langle 0 \rangle$,
\[
I_{n+1}(G^{\bf d}, X)\subseteq \langle \{x_v, \ldots, x_{v^{{\bf d}_v}}\, : \, {\bf d}_v\geq 1\},\{x_{v}+1,\ldots, x_{v^{-{\bf d}_v}}+1\, : \, {\bf d}_v\leq -1\}\rangle \neq \langle 1\rangle
\]
and we get that $\gamma_{\mathcal{P}}(G^{\bf d})\leq n$ for all ${\bf d}\in \mathbb{Z}^{n}$.
\end{proof}

Corollary~\ref{coro:bound} tell us that if we begin with a given graph $G$, then after several duplications (or replications) of a vertex of $G$ its algebraic co-rank can be increased. 
However, in certain point its algebraic co-rank stabilizes.

Next example shows that the upper bound given in Corollary~\ref{coro:bound} is tight.
Moreover, if $\mathrm{det}(L(G,X))|_{X=\phi({\delta})}\in \mathbb{Z}\setminus 0$, then $\gamma_{\mathcal{\mathbb{Z}}}(G^{\delta})= |V(G)|$.
Therefore for any graph $G$, there exists $\delta \in \{1,-1\}^{V(G)}$ such that $\gamma_{\mathcal{\mathbb{Z}}}(G^{\delta})= |V(G)|$.
\begin{Example}\label{example:completa}
Let $K_n$ be the complete graph with $n\geq 2$ vertices.
By \cite[Theorem 3.15]{critical}, we have that $\gamma_{\mathcal{P}}(K_n)= 1$ 
and $I_n(K_n,X)=\langle P\rangle$, where
\[
P=\prod_{j=1}^{n} (x_j+1) - \sum_{i=1}^n \prod_{j\neq i} (x_j+1).
\]
Since the evaluation of $P$ at $\{x_1=0, \cdots, x_{n-1}=0, x_n=-1\}$ is equal to $-1$, by Theorem~\ref{teo:rd} and Corollary~\ref{coro:bound},
$\gamma_{\mathcal{P}}(K_n^{\bf d})=n$ for any ${\bf d}\in \mathbb{Z}^n$ such that ${\bf d}_i\geq 1\text{ for all }1\leq i\leq n-1$ and ${\bf d}_n\leq -1$.
On the other hand, by \cite[Theorem 3.16]{critical}
\[
I_{n-1}(K_n, X)=\left\langle \left\{\prod_{i\in \mathcal{I}} (x_i+1) \, :\, \mathcal{I}\subseteq [n] \text{ and } |\mathcal{I}|=n-2\right\}\right\rangle.
\]
Since $I_{n-1}(K_n, X)_{\{x_i=0\, :\, i\in[n-1]\}}=\langle 1\rangle$, by Theorem~\ref{teo:rd} and Corollary~\ref{coro:bound},
$\gamma_{\mathcal{P}}(K_n^{\bf d})=n-1$ for any ${\bf d}\in \mathbb{Z}^{n}$ such that ${\bf d}_i\geq 1$ for all $1\leq i\leq n-1$.
Note that the graph $K_n^{\bf d}$ for some ${\bf d}\in \mathbb{Z}^{n}$ such that ${\bf d}_i\geq 1$ if and only if $1\leq i\leq j$ for some $1\leq j \leq n-2$ 
is equal to the graph $K_{j+1}^{\bf d'}$ with ${\bf d'}\in \mathbb{Z}^{j+1}$ such that ${\bf d}_i={\bf d}'_i$ for all $1\leq i \leq j$ and ${\bf d}'_{j+1}\leq -1$.
\end{Example}

Corollary~\ref{coro:bound} might be used to construct families of graphs with a fixed algebraic co-rank.
For instance, given a graph $G$ and $\delta \in \{0,1,-1\}^{V(G)}$, let 
\[
\mathcal{T}_{\delta}(G)=\{G^{\bf d}: {\bf d}\in \mathbb{Z}^{|V|} \text{ such that } {\rm supp}({\bf d})=\delta\}.
\]
Then Corollary~\ref{coro:bound} says that the algebraic co-rank of any graph in $\mathcal{T}_{\delta}(G)$ is equal to $k=\gamma_{\mathcal{P}}(G^{\delta})$.
That is, $\mathcal{T}_{\delta}(G)$ is an infinite set of graphs, all of them with algebraic co-rank equal to $k$.
Moreover, these families of graphs are very useful to classify graphs with algebraic co-rank less or equal to a fixed integer.
For instance, in~\cite[Theorem 4.2]{g2} it was proved that a simple connected graph has algebraic co-rank less than or equal to  $2$ 
if and only if it is an induced subgraph of a graph in one of the two families of graphs, $\mathcal{T}_{(1,1,1)}(K_3)$ and $\mathcal{T}_{(-1,1,-1)}(P_3)$, where $K_3$ is the complete graph with three vertices and $P_3$ is the path with three vertices.

On the other hand, one crucial step in the classification of the graphs with algebraic co-rank less than or equal to  $2$
was the use of the concept of a minimal $k$-forbidden graph,
which is a graph with algebraic co-rank greater than or equal to  $k+1$ and it is minimal under induced subgraphs.
In light of the Corollary~\ref{coro:bound}, we have that a minimal $k$-forbidden graph does not have more than three vertices which are twins each other.
Moreover, we conjecture that only a finite number of minimal $k$-forbidden graphs and a finite set $\mathcal{G}$ of pairs $(G,\delta)$ 
of graphs and $\{0,1,-1\}$-vectors exist, such that any graph with algebraic co-rank less than or equal to  $k$ 
is an induced subgraph of a graph in $\bigcup_{(G, \delta)\in \mathcal{G}} \mathcal{T}_{\delta}(G)$.
That is, graphs with twins play an important role on these classification problems.
This opens the question on the distribution of the algebraic co-rank of graphs with twins. 
In~\cite{trees}, a formula for the algebraic co-rank of a tree in terms of its $2$-matching number was given.
Moreover, it can be proved the following.

\begin{Proposition}
If $T$ is a twin-free simple tree, then $\gamma_{\mathcal{P}}(T) \geq \lceil \frac{n+2}{2} \rceil$.
\end{Proposition}
\begin{proof}
It follows by induction on the number of vertices of $T$. 
The smallest twin-free graph is the path with four vertices,
which has algebraic co-rank equal to three, and the rest follows by using that $\gamma_{\mathcal{P}}(T)\geq \gamma_{\mathcal{P}}(T-e)$ for any edge $e$ of $T$, 
see~\cite[Lemma 2.4 and Theorem 3.8]{trees}.
\end{proof}
This result and an intensive computational search of graphs with less than or equal to  nine vertices lead to the following conjecture.

\begin{Conjecture}\label{conj:twin3}
If $G$ is twin-free, then $\gamma_{\mathcal{P}}(G) \geq \lfloor \frac{n}{2} \rfloor$.
\end{Conjecture}

Note that this lower bound for the algebraic co-rank would be tight because the graph with seven vertices given 
in Figure~\ref{fig:2} is twin-free and has algebraic co-rank equal to three.
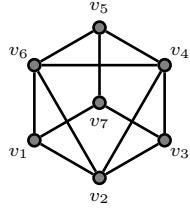
\begin{figure}[h!]
\begin{center}
\begin{tikzpicture}[line width=1pt, scale=1]
\tikzstyle{every node}=[inner sep=0pt, minimum width=4.5pt, circle] 
\draw (30:1) node[draw, fill=gray, label=30:{\tiny $v_4$}] (2) {};
\draw (150:1) node[draw, fill=gray, label=150:{\tiny $v_6$}] (3) {};
\draw (270:1) node[draw, fill=gray, label=270:{\tiny $v_2$}] (7) {};
\draw (330:1) node[draw, fill=gray, label=330:{\tiny $v_3$}] (5) {};
\draw (210:1) node[draw, fill=gray, label=210:{\tiny $v_1$}] (6) {};
\draw (90:1) node[draw, fill=gray, label=90:{\tiny $v_5$}] (1) {};
\draw (0:0) node[draw, fill=gray, label=270:{\tiny $v_7$}] (4) {};
\draw (1) -- (4) -- (5) -- (2) -- (1) -- (3) -- (6) -- (4);
\draw (5) -- (7) -- (2) -- (3) -- (7) -- (6);
\end{tikzpicture}
\end{center}
\caption{A simple graph with seven vertices and algebraic co-rank equal to three.}
\label{fig:2}
\end{figure}

This conjecture is equivalent to the following: if $\gamma_{\mathcal{P}}(G) < \lfloor \frac{n}{2} \rfloor$, then $G$ have at least a pair of twin vertices.
Therefore, if Conjecture~\ref{conj:twin3} is true, then graphs with a low algebraic co-rank have twins and twin-free graphs have an higher algebraic co-rank.

Given $k\geq 1$, let
\[
\Gamma_{\leq k}=\{G\, :\, G \text{ is a simple connected graph with } \gamma(G)\leq k\}.
\]
Then $\Gamma_{\leq 1}=\mathcal{T}_{(-1)}^*(K_1)$, where $\mathcal{T}_{\delta}^*(G)$ 
denotes the set of induced subgraphs of one graph in $\mathcal{T}_{\delta}(G)$.
As we mentioned before, in~\cite{g2} it was proved that 
\[
\Gamma_{\leq 2}=\mathcal{T}_{(1,1,1)}^*(K_3)\cup \mathcal{T}_{(-1,1,-1)}^*(P_3)
\]
and in~\cite{g3} it was given a similar result about $\Gamma_{\leq 3}$.
In general given a fixed constant $k$, we expect that $\Gamma_{\leq k}$ has a similar classification.
Next conjecture goes in this sense.

\begin{Conjecture}\label{conj:finite}
If $k\geq 1$, then 
\[
\Gamma_{\leq k}=\bigcup_{(G, \delta)\in \mathcal{G}} \mathcal{T}_{\delta}^*(G) 
\]
for some finite set $\mathcal{G}$ of pairs $(G, \delta)$ with $G$ being a simple graph and $\delta \in \{0,1,-1\}^{V(G)}$.
\end{Conjecture}

A weak version of Conjecture~\ref{conj:finite} is given by the next conjecture, 
which says that an infinite family of graphs $\{G_i\}_{i=1}^\infty$ with a bounded algebraic co-rank, and such that $G_i$ 
is a proper induced subgraph of $G_j$ for all $i<j$, is essentiality a set of the form $\mathcal{T}_{\delta}(G)$.
\begin{Conjecture}\label{conj:twin4}
If $\mathcal{G}=\{G_i\}_{i=1}^\infty$ is an infinite family of simple graphs such that $G_i$ is a proper induced subgraph of $G_j$ for all $i<j$, then either
\[
{\rm max}\{\gamma_{\mathcal P}(G_i)\}_{i=1}^\infty=\infty
\] 
or a graph $G$, a vector $\delta\in \{0,1,-1\}^{V(G)}$ and an integer $M\in \mathbb{N}$ such that $G_i\in \mathcal{T}_{\delta}(G)$ for all $i\geq M$ exist.
\end{Conjecture}

There are families of graphs with unbounded algebraic co-rank.
For instance, if $\mathcal{G}=\{P_i\}_{i=1}^\infty$ where $P_i$ is the path with $i$ vertices, 
then $\gamma_{\mathcal P}(P_i)=i-1$ and therefore ${\rm max}\{\gamma_{\mathcal P}(P_i)\}_{i=1}^\infty=\infty$.

Now, we prove that Conjecture~\ref{conj:twin3} implies Conjectures~\ref{conj:finite} and~\ref{conj:twin4}.

\begin{Theorem}\label{equivalence1}
Conjecture~\ref{conj:twin3} implies Conjecture~\ref{conj:twin4}.
\end{Theorem}
\begin{proof}
Let $\mathcal{G}=\{G_i\}_{i=1}^\infty$ be as in Conjecture~\ref{conj:twin4} with $\gamma={\rm max}\{\gamma_{\mathcal P}(G_i)\}_{i=1}^\infty< \infty$.

Our strategy is to give a lower bound for the algebraic co-rank of a graph.
The modular decomposition of a connected graph is obtained from a prime graph, that is, a twin-free graph, $K_2$ or $K_1$, by blowing-up each vertex with a cograph.
In this way, Conjecture~\ref{conj:twin3} will allow us to reduce Conjecture~\ref{conj:twin4} to the case of a cograph.
Before that, let us explain this assertion and recall the definition of modular decomposition of a graph and the concept of cograph and its cotree.
Given a graph $G$, a module of $G$ is a subset $U$ of their vertices, such that
\[
N_{G\setminus U}(u)=N_{G\setminus U}(v) \text{ for all } u,v\in U,
\]
where $N_{G}(u)$ is the open neighborhood of $u$ in $G$, that is the set of neighborhoods of the vertex $u$ in $G$.
For instance, the set of true (or false) twins of a given vertex is a module.
Note that one module can be a subset of another module.
Also note that the entire set of vertices of $G$ and any single vertex of $G$ is a module of $G$, which are called {\it trivial modules}.
The modular decomposition of a graph $G$ consists in decomposing the vertex set $V(G)$ in their modules.
That is, the modular decomposition is a recursive and hierarchical decomposition of the graph, not only a partition of their vertices.
For simple graphs, this decomposition is unique.
A graph is called prime if all their modules are trivial.
Note that a graph different from $K_2$ or $K_1\sqcup K_1$ is prime if and only if it is a twin-free graph.
This hierarchical decomposition can be encoded into a tree where the root corresponds to the maximal induced prime subgraph and leaves the vertices of $G$ and the other internal vertices labeled with join ($\boxtimes$) or disjoin union ($\sqcup$) operations 
(similar to replications and duplications to modules). See Example~\ref{examplemodular}.

\begin{Example}\label{examplemodular}
The graph in Figure~\ref{modular} has four nontrivial modules $\{v_1,v_5,v_6\}$, $\{v_5,v_6\}$,$\{v_3,v_7\}$ and $\{v_4,v_8\}$.
The induced subgraphs of these modules are cographs.
Its maximal induced prime subgraph is the path with four vertices $P_4=u_1u_2u_3u_4$.
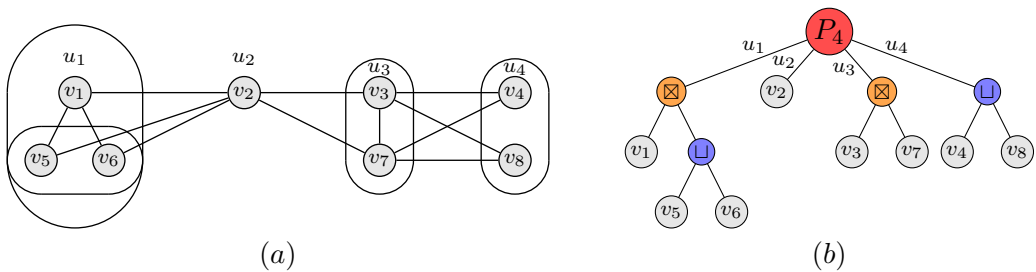
\begin{figure}[h]
\begin{center}
\begin{tabular}{c@{\extracolsep{10mm}}c}
\begin{tikzpicture}[line width=0.5pt, scale=0.9]
	\tikzstyle{every node}=[minimum width=0pt, inner sep=1pt, circle, draw, fill=gray!20!]
	\draw (-3.5,0) node (v1) {\scriptsize$v_1$};
	\draw (-1,0) node (v2) {\scriptsize $v_2$};
	\draw (1,0) node (v3) {\scriptsize $v_3$};
	\draw (3,0) node (v4) {\scriptsize $v_4$};
	\draw (-4,-1) node (v1p) {\scriptsize $v_5$};
	\draw (-3,-1) node (v2p) {\scriptsize $v_6$};
	\draw (1,-1) node (v3p) {\scriptsize $v_7$};
	\draw (3,-1) node (v4p) {\scriptsize $v_8$};
	\tikzstyle{every node}=[] 
	\draw (-3.5,0.5) node () {\scriptsize$u_1$};
	\draw (-1,0.5) node () {\scriptsize$u_2$};
	\draw (1,0.35) node () {\scriptsize$u_3$};
	\draw (3,0.35) node () {\scriptsize$u_4$};
	\draw (v1) -- (v2) -- (v3) -- (v4);
	\draw (v1p) -- (v2) -- (v3p) -- (v4p);
	\draw (v1p) -- (v1) -- (v2p) -- (v2);
	\draw (v3) -- (v3p);
	\draw (v3) -- (v4p);
	\draw (v3p) -- (v4);
  	\draw (-2.5,0) arc (0:180:1);
  	\draw (-2.5,0) -- (-2.5,-1);
	\draw (-4.5,0) -- (-4.5,-1);
	\draw (-4.5,-1) arc (180:360:1);
	
	\draw (-4,-0.5) arc (90:270:0.5);
  	\draw (-4,-0.5) -- (-3,-0.5);
	\draw (-4,-1.5) -- (-3,-1.5);
	\draw (-3,-1.5) arc (-90:90:0.5);
	
	\draw (1.5,0) arc (0:180:0.5);
  	\draw (0.5,0) -- (0.5,-1);
	\draw (1.5,0) -- (1.5,-1);
	\draw (0.5,-1) arc (180:360:0.5);
	\draw (3.5,0) arc (0:180:0.5);
  	\draw (2.5,0) -- (2.5,-1);
	\draw (3.5,0) -- (3.5,-1);
	\draw (2.5,-1) arc (180:360:0.5);
\end{tikzpicture}
&
\begin{tikzpicture}
[level distance=8mm,
level 1/.style={sibling distance=14mm}, 
level 2/.style={sibling distance=8mm}, 
level 3/.style={sibling distance=8mm}]
\tikzstyle{every node}=[minimum width=0pt, inner sep=1pt, circle, draw, fill=gray!20!] 
\node[fill=red!70!] (root) {$P_4$}
	child{
		child 
		child{
			child
			child
			}
		}
	child
	child{
		child
		child
		}
	child{
		child
		child
		};
\node[fill=orange!70!] at (root-1) {\scriptsize $\boxtimes$};
\node at (root-2) {\scriptsize $v_2$};
\node[fill=orange!70!] at (root-3) {\scriptsize $\boxtimes$};
\node[fill=blue!50!] at (root-4) {\scriptsize $\sqcup$};
\node at (root-1-1) {\scriptsize $v_1$};	
\node[fill=blue!50!] at (root-1-2) {\scriptsize $\sqcup$};
\node at (root-1-2-1) {\scriptsize $v_5$};
\node at (root-1-2-2) {\scriptsize $v_6$};
\node at (root-3-1) {\scriptsize $v_3$};
\node at (root-3-2) {\scriptsize $v_7$};
\node at (root-4-1) {\scriptsize $v_4$};
\node at (root-4-2) {\scriptsize $v_8$};
\tikzstyle{every node}=[] 
\draw (-1,-0.2) node () {\scriptsize$u_1$};
\draw (-0.6,-0.4) node () {\scriptsize$u_2$};
\draw (0.2,-0.5) node () {\scriptsize$u_3$};
\draw (0.9,-0.2) node () {\scriptsize$u_4$};
\end{tikzpicture}
\\
$(a)$
&
$(b)$
\end{tabular}
\end{center}
\caption{A modular decomposition of a graph where each of the ellipses correspond to an non-trivial module and its associated tree.}
\label{modular}
\end{figure}
\end{Example}

Now, let us introduce the concept of cograph.
There exist several alternative definitions of a cograph; one of them says that a cograph is a simple graph without the path $P_4$ with four vertices as an induced subgraph.
Another characterization of a cograph says that a cograph is a graph in which every nontrivial induced subgraph has at least a pair of twins.
That is, the graph corresponding to the root in the tree associated to the modular decomposition of a cograph is equal to $K_1$.
Note that the non-trivial module of a graph induces a cograph.
Moreover, the modular decomposition of a graph decomposes it as a twin-free graph and a blow-up 
(replace a vertex $v$ with a graph $H$ such that $N(u)\cap V(G)=N_G(v)$ for all $u\in H$) of each of their vertices with a cograph, see Figure~\ref{modular}.
The reader may consult~\cite{cographs} and the references contained there for more details about cographs and its cotrees.

Now, we will first use Conjecture~\ref{conj:twin3} to reduce Conjecture~\ref{conj:twin4} to the case of a cograph.
Let
\[
\mathcal{T}=\{ G\, | \, G \text{ is a twin-free induced subgraph of } G_i \text{ for some } G_i \in \mathcal{G}\}
\]
be the set of all twin-free graphs that are induced subgraphs of some $G_i \in \mathcal{G}$.
Since the trivial graph with one vertex is twin-free, $\mathcal{T}$ is non-empty.
Moreover, Conjecture~\ref{conj:twin3} implies that any graph in $\mathcal{T}$ has at most $2\gamma+1$ vertices
and thus $\mathcal{T}$ is finite.
Let $G$ be maximal (under the partial order given by induced subgraphs) graph in $\mathcal{T}$ 
and $M$ be the first natural number such that $G$ is an induced subgraph of $G_M$.
For any $v\in V(G)$, let 
\[
L_v^j=\{ u \, | \, u \in V(G_j) \text{ and such that } N_G(v)=N_{G_j}(u)\cap (V(G)-v)\}.
\]
Clearly, $v\in L_v^j$ for all $v\in V(G)$ and $j\geq M$.
Since $G$ is maximal, for any $u\in G_j-G$ the subgraph of $G_j$ induced by $V(G)\cup u$ is not twin-free.
Also, since $G$ is twin-free, $L_u^j\cap L_v^j=\emptyset$ for all $u\neq v\in V(G)$ and for all $j$.
That is, any $u\in G_j-G$ belongs to $L_v^j$ for some $v\in V(G)$ and thus $\bigsqcup_{v\in V(G)}L_v^j=V(G_j)$.
Note that the $L_v^j$ are the maximal modules of $G_j$ and $G$ is the maximal prime subgraph of $G_j$.

All the vertices in $L_v^j$ play the same role of $v$ in the sense that if $u\in L_v^j$, then the induced subgraph 
by $(V(G)-v)\cup u$ in $G_j$ is $G$.
Moreover, for any $v\in V(G)$ and $j\geq M$, the subgraph $G_j[L_v^j]$ of $G_j$ induced by $L_v^j$ is a cograph.
Otherwise, $G_j[L_v^j]$ would contain $P_4$ as an induced subgraph and therefore the subgraph of $G_j$ 
induced by the union of vertices of $G-v$ and the vertices of $P_4$ would be a twin-free graph; 
a contradiction to the maximality of $G$ (any vertex in $P_4$ play the same role of $v$).
Note that this gives us the modular decomposition of $G_j$ as a twin-free graph with the blow-up (with a cograph) in each of their vertices.
 
Until now, using Conjecture~\ref{conj:twin3} we have proved that the maximal twin-free subgraph (or maximal prime subgraph) of the ${G_j}'s$ stabilizes at some point.
Therefore we need to prove that their maximal modules also must be stabilized 
(in the sense that there exists a graph $H$, $\delta\in \{0,1,-1\}^{V(H)}$, and $M\in \mathbb{N}$ such that $G_i[L_v^i]\in \mathcal{T}_{\delta}(H)$ for all $i\geq M$) at some point.
Since the modules are cographs, without loss of generality we may assume that
$\mathcal{G}=\{G_i\}_{i=1}^\infty$ (taking $\mathcal{G}=\{G_i[L_v^i]\}_{i=1}^\infty$) consists of cographs.
Given a cograph, its cotree is the tree obtained through its modular decomposition.
We will use this cotree to bound the algebraic co-rank of its cograph.
Let $C$ be a cograph and $T$ be its cotree.
Let $\widetilde{T}$ be the tree obtained by erasing the leaves of $T$, and $T'$ be the tree obtained by erasing the twins of $T$.
Note that every twin of $T$ is a leaf, but not every leaf has a twin.
For instance, in Figure~\ref{cotree}.$(b)$ the vertex $v_3$ is a leaf with no twins.
Note that, 
\[
C=(H)^{\bf d} \text{ for some }{\bf d}\in \mathbb{Z}^{V(H)},
\]
where $H$ is the cograph with cotree equal to $T'$, see Figure~\ref{cotree} $(c)$ and $(d)$.
Moreover, twin vertices in $T$ correspond to twin vertices in $C$.
The next example illustrate this situation.

\begin{Example}
Figure~\ref{cotree} contains a cograph $C$, its cotree $T$, the cotree $T'$ obtained by erasing its twins and its corresponding graph $H$.
Clearly $C=H^{(1,0,-1,-1)}$.
\begin{figure}[h]
\begin{center}
\begin{tabular}{c@{\extracolsep{5mm}}c@{\extracolsep{5mm}}c@{\extracolsep{5mm}}c}
\begin{tikzpicture}[line width=0.5pt, scale=0.9]
	\tikzstyle{every node}=[minimum width=0pt, inner sep=1pt, circle, draw, fill=gray!20!]
	\draw (-0.7,0) node (v1) {\scriptsize $v_1$};
	\draw (0.7,0) node (v2) {\scriptsize $v_2$};
	\draw (-2.5,-2) node (v3) {\scriptsize $v_3$};
	\draw (-0.7,-3) node (v4) {\scriptsize $v_4$};
	\draw (0.7,-3) node (v5) {\scriptsize $v_5$};
	\draw (2,-2.5) node (v6) {\scriptsize $v_6$};
	\draw (2.8,-1.8) node (v7) {\scriptsize $v_7$};
	\tikzstyle{every node}=[] 
	\draw (0,-1) node {$C$};
	\draw (v4) -- (v5);
	\draw (v1) -- (v3);
	\draw (v1) -- (v4);
	\draw (v1) -- (v5);
	\draw (v1) -- (v6);
	\draw (v1) -- (v7);
	\draw (v2) -- (v3);
	\draw (v2) -- (v4);
	\draw (v2) -- (v5);
	\draw (v2) -- (v6);
	\draw (v2) -- (v7);
	\draw (v6) -- (v7);	
\end{tikzpicture}
&
\begin{tikzpicture}
[level distance=8mm,
level 1/.style={sibling distance=26mm}, 
level 2/.style={sibling distance=12mm}, 
level 3/.style={sibling distance=6mm}]
\tikzstyle{every node}=[minimum width=0pt, inner sep=1pt, circle, draw, fill=gray!20!]
\node[fill=orange!70!] (root) {\scriptsize $\boxtimes$}
	child{
		child 
		child
		}
	child{
		child
		child{
			child
			child
			}
		child{
			child
			child
			}
		};
\node at (root-2-2-1) {\scriptsize $v_4$};
\node at (root-2-2-2) {\scriptsize $v_5$};
\node at (root-2-3-1) {\scriptsize $v_6$};
\node at (root-2-3-2) {\scriptsize $v_7$};
\node at (root-1-1) {\scriptsize $v_1$};	
\node at (root-1-2) {\scriptsize $v_2$};
\node at (root-2-1) {\scriptsize $v_3$};
\tikzstyle{every node}=[minimum width=0pt, inner sep=1pt, circle, draw, fill=blue!50!] 
\node at (root-1) {\scriptsize $\sqcup$};
\node at (root-2) {\scriptsize $\sqcup$};
\tikzstyle{every node}=[minimum width=0pt, inner sep=1pt, circle, draw, fill=orange!70!] 
\node  at (root-2-2) {\scriptsize $\boxtimes$};
\node at (root-2-3) {\scriptsize $\boxtimes$};
\tikzstyle{every node}=[] 
\draw (0,-1) node {$T$};
\end{tikzpicture}
&
\begin{tikzpicture}
[level distance=8mm,
level 1/.style={sibling distance=12mm}, 
level 2/.style={sibling distance=8mm}, 
level 3/.style={sibling distance=6mm}]
\tikzstyle{every node}=[minimum width=0pt, inner sep=1pt, circle, draw, fill=gray!20!]
\node[fill=orange!70!]  (root) {\scriptsize $\boxtimes$}
	child
	child{
		child
		child
		child
		};
\node at (root-1) {\scriptsize $u_1$};
\node[fill=blue!50!]  at (root-2) {\scriptsize $\sqcup$};
\node at (root-2-1) {\scriptsize $u_2$};
\node at (root-2-2) {\scriptsize $u_3$};
\node at (root-2-3) {\scriptsize $u_4$};
\tikzstyle{every node}=[] 
\draw (0,-1) node {$T'$};
\draw (0,-2.5) node {\mbox{ }};
\end{tikzpicture}
&
\begin{tikzpicture}[line width=0.5pt, scale=0.9]
	\tikzstyle{every node}=[minimum width=0pt, inner sep=1pt, circle, draw, fill=gray!20!]
	\draw (0,0) node (u1) {\scriptsize $u_1$};
	\draw (-1,-2) node (u2) {\scriptsize $u_2$};
	\draw (0,-2) node (u3) {\scriptsize $u_3$};
	\draw (1,-2) node (u4) {\scriptsize $u_4$};
	\draw (u1) -- (u2);
	\draw (u1) -- (u3);
	\draw (u1) -- (u4);
	\tikzstyle{every node}=[] 
	\draw (-0.2,-1) node {$H$};	
\end{tikzpicture}
\\
$(a)$
&
$(b)$
&
$(c)$
&
$(d)$
\end{tabular}		
\end{center}
\caption{A cograph $C$, its cotree $T$, a cotree $T'$ obtained from $T$ and the cograph associated to $T'$.}
\label{cotree}
\end{figure}
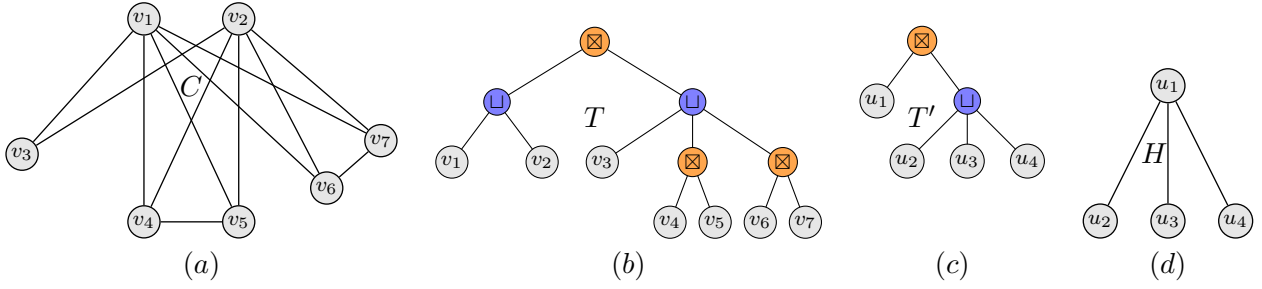
\end{Example}

Now, we give a lower bound for the algebraic co-rank of a connected cograph $C$ in function of the height of its cotree $T$ 
and the out-degree of the vertices of $\widetilde{T}$.
Before we do that, we give a lower bound for the algebraic co-rank of a special class of cographs: the threshold graphs.
Let $Th_1$ be the trivial graph with only one vertex (denoted by $v_1$) and
\[
Th_n=
\begin{cases}
v_{2k}\boxtimes Th_{2k-1} & \text{ if } n=2k,\\
v_{2k+1}\sqcup Th_{2k}  & \text{ if } n=2k+1,
\end{cases}
\]
where $\boxtimes$ means the join of graphs and $\sqcup$ means the disjoint union of graphs.
Since 
\[
L(Th_{2k},X)[\{2,4,\ldots,2k\},\{1,3,\ldots,2k-1\}]=
\left[\begin{array}{cccc}
-1 &    0 &   \cdots &    0 \\
-1 &    -1 &     0 &    \vdots \\
\vdots &    &    \ddots &    0 \\
-1 &     &     \cdots &    -1
\end{array}\right],
\]
$\gamma_{\mathcal P}(Th_{2k})\geq k$.
Note that any connected threshold graph different from the trivial graph 
belongs to one of the families of graphs $\mathcal{T}_{(1,-1,\ldots,1,-1)}(Th_{2k})$ for all $k\geq 1$.
In a similar way, it can be proved that $\gamma_{\mathcal P}(\boxtimes_{i=1}^l Th_{2k_i})\geq l-1+\sum_{i=1}^l k_i$.

Now, let $r$ be the root of $T$, $u$ be one of their leaves, and consider the path $P_{u}(T)$ from $u$ to $r$.
Note that if $C$ is connected, then the root of $T$ is labeled with a join operation.
If $P_{u}(T)$ has length $n$ (the number of edges), then $C$ contains the threshold graph $Th_{n+1}$ 
as an induced subgraph when $n$ is odd and $Th_{n}$ when $n$ is even.
For instance, consider the cograph given in Figure~\ref{cotree}.$(a)$ and its cotree given in Figure~\ref{cotree}.$(b)$.
If we choose the vertex $v_4$, then the length of $P_{v_4}(T)$ is three, and the graph induced by $v_4,v_5,v_6,v_1$ is equal to $Th_4$.
In a similar way, if we choose $v_1$, then the length of $P_{v_1}(T)$ is two, and it can be check that the graph induced by $v_1,v_3$ is equal to $Th_2$.
Now, if $T$ has height $h$, then $C$ contains $Th_{h}$ as an induced subgraph.
Since $\gamma_{\mathcal P}(C)\geq\gamma_{\mathcal P}(Th_{h})\geq \lfloor \frac{h}{2}\rfloor$ and $\gamma_{\mathcal P}(G_j)=\gamma < \infty$,
the height of $T$ is less or equal to $2\gamma+1$.

Now, we give a lower bound for the algebraic co-rank of a connected cograph $C$ in terms of the 
out-degree of the vertices of its associated tree $\widetilde{T}$, (the obtained from its cotree by erasing its leaves).
Let $u$ be a vertex of $\widetilde{T}$.
If $u$ is a leaf of $\widetilde{T}$, then its out-degree is zero.
Assume that $u$ is not a leaf of $\widetilde{T}$.
If $u$ is labeled with a disjoint union operation, then $C$ contains a disjoint union of the subgraphs associated to the out-neighborhoods of $u$.
Since
\[
\gamma_{\mathcal P}\left(\bigsqcup_{i=1}^n H_i\right)=\sum_{i=1}^n \gamma_{\mathcal P}(H_i),
\] 
and the unique graph with algebraic co-rank equal to zero is the trivial graph, the algebraic co-rank of $C$ is at least the out-degree of $u$.
If $u$ is labeled with a join operation, then $C$ contains the join of the subgraphs associated to the out-neighborhoods of $u$.
On the other hand, if $\{H_i\}_{i=1}^l$ is a set of graphs, each one different to a complete graph, then
$\boxtimes_{i=1}^l H_i$ contains $K_l^{(1,\ldots,1)}$ and therefore by Example~\ref{example:completa}
\[
\gamma_{\mathcal P}(\boxtimes_{i=1}^l H_i)\geq \gamma_{\mathcal P}(K_l^{(1,\ldots,1)})=l-1.
\]
Thus the algebraic co-rank of $C$ is at least the out-degree of $u$ minus one.
Since $\gamma_{\mathcal P}(G_j)=\gamma < \infty$, the out-degree of the vertices of $\widetilde{T}$ is less or equal to $\gamma+1$.
Therefore the height and the out-degree of all the vertices non adjacent to a leaf of $T$ are bounded.

Since $G_i$ is a proper induced subgraph of $G_j$ for all $i<j$ and any induced subgraph of a cograph $C$ corresponds to a subcotree of the cotree $T$ of $C$
(the subcotree of $T$ induced by the leaves that correspond to the vertices in the induced subgraph and rooted by the common ancestor of these leaves), 
there exists a cograph $C_v$ and $M'\geq M$ such that 
\[
G_j[L_v^j]=C_v^{\bf d}\text{ for some }{\bf d}\in \{0,1,-1\}^{V(C_v)}\text{ for all }j\geq M'
\] 
and therefore we get the statement given in Conjecture~\ref{conj:twin4}.
\end{proof}

Note that any lower bound for the algebraic co-rank of a graph in terms of its number of vertices in the case 
of twin-free graphs and in terms of the structure of its cotree in case of cographs implies Conjectures~\ref{conj:finite} and~\ref{conj:twin4}.
A key fact in the proof of Theorem~\ref{equivalence1} is to give a lower bound for the algebraic co-rank of a cograph in terms of its cotree.
The lower bound presented in the proof of Theorem~\ref{equivalence1} is very loose, however we conjecture the following:

\begin{Conjecture}\label{boundcograph}
If $C$ is a cograph, then
\[
\gamma_{\mathcal P}(C)\geq |E(\widetilde{T})|-\#\{\text{internal vertices of } \widetilde{T} \text{ labeled with the join or join operatation}\},
\]
where $\widetilde{T}$ is the tree obtained from the cotree of $C$  by erasing their leaves.
\end{Conjecture} 

As a consequence of the lower bound for the algebraic co-rank of a cograph given in the proof of Theorem~\ref{equivalence1}, we have the following result:
\begin{Corollary}\label{finitecograph}
If $k$ is a positive integer, then
\[
\mathcal{C}_{\leq k}=\{C\, | \, C \text{ is a cograph in } \Gamma_{\leq k} \}=\bigcup_{(G, \delta)\in \mathcal{G}} \mathcal{T}_{\delta}^*(G) 
\]
for some finite set $\mathcal{G}$ of pairs $(G, \delta)$ with $\delta \in \{0,1,-1\}^{V(G)}$.
\end{Corollary}
\begin{proof}
Let $C$ be a cograph with algebraic co-rank less than or equal to  $k$, $T$ be its cotree, $T'$ be the tree obtained from $T$ by erasing their twin vertices, and $H$ be the cograph with cotree equal to $T'$.
Clearly $C=H^{\bf d}$ for some ${\bf d}\in \mathbb{Z}^{V(H)}$.
From the proof of Theorem~\ref{equivalence1}, we have that the height of $T$, which is also the height of $T'$, 
is upper bounded and the out-degree of the vertices of $T'$ is also upper bounded.
Therefore, there exists a finite number of trees that can be the $T'$ of $C$ and therefore it turns out the result.
\end{proof}

\begin{Theorem}\label{equivalence2}
Conjecture~\ref{conj:twin3} implies Conjecture~\ref{conj:finite}.
\end{Theorem}
\begin{proof}
Let $G$ be a graph with algebraic co-rank less than or equal to  $k$.
Using the modular decomposition of $G$ we have that $G$ can be decompose in a twin-free (its maximal prime) graph 
with a blow-up of a cograph in each of their vertices.
By Conjecture~\ref{conj:twin3} we have that the size of its maximal twin-free is bounded 
(therefore only a finite number of possible maximal twin-free subgraphs for $G$ exist)
and by Corollary~\ref{finitecograph} only a finite number of cographs $H$ exist, 
such that the cograph used in the blow-up of each vertex is of the form $H^{\bf d}$ for some ${\bf d}\in \mathbb{Z}^{V(H)}$.
Putting together these two facts we get the result.
\end{proof}

Finally, we pose a variant of Conjecture~\ref{conj:twin3}, which we believe is stronger.

\begin{Conjecture}\label{conj:twin1} 
If $\gamma_{\mathcal{P}} (G - v ) = \gamma_{\mathcal{P}} (G)$ for all $v \in V (G)$, then $G$ has at least a pair of twin vertices.
\end{Conjecture}


\section{Critical ideals of graphs with twin vertices}\label{sec:description}

In this section we give a deep description of some of the critical ideals of a graph $G$ obtained by 
duplicating or replicating several times one of their vertices in terms of some of the critical ideals of $G$.
In Section~\ref{rd} we saw that the algebraic co-rank of the graphs $\{d^k(G,v)\}_{k\geq 0}$ and $\{r^k(G,v)\}_{k\geq 0}$ quickly stabilizes.
We will show that their critical ideals regularize, but a little bit slower.
More precisely, if $\gamma_d=\gamma_{\mathcal{P}}(d(G,v))$ and $\lambda\in \{0,1\}$ is a constant that depends on $G$ and $v$, 
then Theorem~\ref{teo:deq} gives a description of 
\[
I_{\gamma_d+k}(d^{k+\lambda+i}(G,v),X)
\]
in terms of the critical ideals of $G$. 
Also, Theorem~\ref{teo:req} gives a similar description of the critical ideals of $I_{\gamma_r+k}(r^{k+\lambda+i}(G,v),X)$, 
where $\gamma_r=\gamma_{\mathcal{P}}(r(G,v))$.

\subsection{The critical ideals of the duplication of vertices}
We begin by giving a description of the critical ideals of $d^k(G,v)$ in terms of the critical ideals of $G$ and some of the minors of $G-v$.
This description generalizes the description of the critical ideals of $d(G,v)$ given in Equation~\ref{eq:d}.

Before doing this, we need to introduce some notation.
Given a subset $S$ of natural numbers, $\binom{S}{l}$ denote the set of all subsets of $S$ of cardinality equal to $l$.
Moreover, if $v$ is a vertex of a signed multidigraph, let 
\[
P^S_{l}(v)=\left\{\prod_{c\in C}x_{v^c}\, : \, C\in \binom{S}{l}\right\},
\]
that is, $P^S_{l}(v)$ is the set of the products of $l$ of the variables associated to the duplication of one vertex of $G$.
By convention we take $P^S_{0}(v)=\{1\}$.
And for simplicity, $P^k_{l}(v)$ denote $P_{l}^{\{0,\ldots ,k\}}(v)$.
Note that $I_{l}(T_k,X)=\langle P^k_{l}(v)\rangle$, where $T_k$ is the trivial graph with $k$ isolated vertices and $v$ is a vertex of $T_k$.
We also recall that $I_{j}(G, X)=\langle 0\rangle$ for all $j> |V(G)|$.

\begin{Lemma}\label{lema:gend}
Let $G$ be a signed multidigraph with $n\geq 2$ vertices and $v\in V(G)$.
If $k,j\geq 1$ and $m={\rm min}(k,j-1)$, then
\begin{eqnarray*}
I_{j}(d^k(G,v),X)&=&\left\langle  \left\{ P_{l}^{k}(v) \cdot I_{j-l}(G, X)|_{x_v=0} \right\}_{l=0}^{m-1}, P_{m}^{k}(v) \cdot I_{j-m}(G- v,X),\right.\\
&& P_{m}^{k}(v)\cdot {\rm minors}_{j-m}({\bf a},L(G- v,X)), \left. P_{m}^{k}(v)\cdot {\rm minors}_{j-m}(L(G- v,X), {\bf b}), S_{j}^{k}(G,v) \right\rangle
\end{eqnarray*}
for all $1\leq j\leq n+k$,
where $S_{j}^{k}(G,v)$ is equal to $P^k_{j}(v)$ when $j\leq k+1$, and equal to
\[
\left\{ {\rm det}(M)\cdot \prod_{t=0}^{k} x_{v^{t}}+{\rm det}(J(0,{\bf a'};M,{\bf b'}))\cdot\sum_{t=0}^{k} \prod_{s\neq t} x_{v^{s}} \, : \, J(x_v,{\bf a'};M,{\bf b'}) \in M_{j-k}(L(G,X))\right\}
\]
when $j> k+1$.
\end{Lemma}
The proof of this lemma is technical and very similar to those arguments given in the previous proofs. 
It is included in Section 3.4 at the end of this section.

\begin{Remark}
Note that $I_{j}(G, X)|_{x_v=0}$ is equal to
\[
\langle {\rm minors}_{j}(L(G- v,X)), {\rm minors}_{j}({\bf a}, L(G- v,X)),
 {\rm minors}_{j}(L(G- v,X), {\bf b}), {\rm minors}_{j}({\bf a},L(G- v,X), {\bf b}) \rangle
\]
and the $i$-{\it th} critical ideal $I_i(T_{k+1},X)$ of the graph with $k+1$ isolated vertices is equal to $\langle P_{i}^{k}(v)\rangle$.
Moreover, if $m={\rm min}(k,j-1)$, then 
\[
I_j(d^k(G,v),X)|_{x_{v}=0}=\left\langle \left\{ P_{i}^{\{1,\ldots,k\}}(v) \cdot I_{j-i}(G, X)|_{x_v=0} \right\}_{i=0}^{m} \right\rangle. 
\]
By~\cite[Proposition 3.4]{critical} the $j$-{\it th} critical ideal of the disjoint union of $T_{k+1}$ and $G$ is equal to 
\[
I_j(T_{k+1}\sqcup G,X)=\left\langle\, \bigcup_{i=0}^j I_i(T_{k+1},X)\cdot I_{j-i}(G,X)\,\right\rangle=\left\langle\, \bigcup_{i=0}^j P_{i}^{k}(v)\cdot I_{j-i}(G,X)\,\right\rangle.
\]
That is, $I_j(d^k(G,v),X)|_{x_{v}=0}$ behaves almost equal as the $j$-{\it th} critical ideal of the disjoint union of $T_{k+1}$ and $G$. 
\end{Remark}

In the next example we show how to use the description of $I_j(d^k(G,v),X)|_{x_{v}=0}$.
\begin{Example}
Let $Q_3$ be the hypercube with $V(Q_3)=\{v_i\}_{i=1}^8$. 
The reader can check that $\gamma_{\mathbb Z}(Q_3)=4$, $\gamma_{\mathbb Z}(d(Q_3,v_8))=5$.
Moreover 
\[
I_7(d(Q_3,v_8),X)|_{x_{8}=0}=\langle x_{8^1}\cdot I_6(Q_3,X)|_{x_8=0}, I_7(Q_3,X)|_{x_8=0}\rangle,
\] 
where $I_6(Q_3,X)_{x_8=0}=\left\langle x_1-x_6, x_2-3x_7, x_3-x_6, x_4-x_7, x_5-x_7, x_6x_7-1\right\rangle$ and
\begin{eqnarray*}
I_7(Q_3,X)|_{x_8=0}&=&\langle
x_2x_4x_6-x_4x_5x_6-x_4x_6x_7-x_5x_6x_7-x_2-x_4+2x_5+2x_7,\\
&&
x_2x_3x_5-x_3x_4x_5-x_3x_4x_7-x_3x_5x_7-x_2+2x_4-x_5+2x_7,\\
&&
x_1x_2x_7-x_1x_4x_5-x_1x_4x_7-x_1x_5x_7-x_2+2x_4+2x_5-x_7,\\
&&
x_1x_3x_7-x_1x_4x_6+x_3x_4x_6-x_1x_6x_7+x_1-2x_3+x_6,\\
&& 
x_1x_3x_5+x_1x_4x_6-x_3x_4x_6-x_3x_5x_6-2x_1+x_3+x_6,\\
&&
x_1x_4x_5x_6+x_1x_4x_6x_7+x_1x_5x_6x_7-x_1x_5-x_5x_6-2x_4x_6-2x_1x_7+3,\\
&&
x_3x_4x_5x_6+x_3x_4x_6x_7+x_3x_5x_6x_7-2x_3x_5-2x_4x_6-x_3x_7-x_6x_7+3\rangle.
\end{eqnarray*}
\end{Example}

We now are ready to give a more accurate description of some critical ideals of $d^{i+k}(G,v)$.
Given $r,s\geq 0$, let
\[ 
\lambda(r,s)=
\begin{cases}
0 & \text{ if } r=s,\\
1 & \text{ otherwise}.
\end{cases}
\]
As the next theorem shows, this constant plays the role of a regularity constant in the sense that the behavior of the critical ideals 
$I_{\gamma_d+k}(d^{k+\lambda}(G,v),X)$ is regular.
Moreover, $\lambda=0$ if and only if $\gamma_{\mathcal P}(G- v)=\gamma_{\mathcal P}(d(G,v))$,
that is, $\lambda$ indicates whether the removal of $v$ or the duplication of $v$  yields a change in the algebraic co-rank.

\begin{Theorem}\label{teo:deq}
Let $G$ be a signed multidigraph, $v$ a vertex of $G$, $\gamma_d=\gamma_{\mathcal P}(d(G,v))$, $\gamma_v=\gamma_{\mathcal P}(G- v)$
and $\lambda=\lambda(\gamma_{v},\gamma_d)$.
If $\gamma_{\mathcal P}(G) \geq 2$, then $0\leq \gamma_d-\gamma_v\leq 2$ and
\[
I_{\gamma_d+k}(d^{k+\lambda+i}(G,v),X)= \left\langle  \left\{ P_{l}^{k+\lambda+i}(v) \cdot I_{\gamma_d+k-l}(G, X)|_{x_v=0} \right\}_{l=0}^{k}\right\rangle 
\]
for all $k\geq 1$ and $i\geq 0$.
\end{Theorem}
\begin{proof}
Since $I_j(G,X)|_{x_v=0}\subseteq I_{j-2}(G- v, X)$, by Lemma~\ref{lema:d} we have that $0\leq \gamma_d-\gamma_v\leq 2$.
Note that $\gamma_d-\gamma_v$ measures the number of steps in which the algebraic co-rank of the set of graphs $\{G-v,d^k(G,v)_{k\geq 0}\}$ stabilizes.
The inequality $0\leq \gamma_d-\gamma_v\leq 2$ says that this happens in at most two steps.
Now, applying Lemma~\ref{lema:gend} with $k=k+\lambda+i$ and $j=\gamma_d+k$,
we get that $I_{\gamma_d+k}(d^{k+\lambda+i}(G,v),X)$ is equal to 
\begin{gather*}
\left\langle \left\{ P_{l}^{k+\lambda+i}(v) \cdot I_{\gamma_d+k-l}(G, X)|_{x_v=0} \right\}_{l=0}^{m-1}, P_{m}^{k+\lambda+i}(v) \cdot I_{\gamma_d+k-m}(G- v, X),\right.\\
\left. P_{m}^{k+\lambda+i}(v)\cdot {\rm minors}_{\gamma_d+k-m}({\bf a},L(G- v, X)),P_{m}^{k+\lambda+i}(v)\cdot {\rm minors}_{\gamma_d+k-m}(L(G- v, X), 
S_{\gamma_d+k}^{k+\lambda+i}(G,v)\right\rangle.
\end{gather*}
On the other hand, since $\gamma_d-1\geq \lambda$, $m={\rm min}(k,j-1)={\rm min}(k+\lambda+i,\gamma_d+k-1)=
k+{\rm min}(\lambda+i,\gamma_d-1)\geq k+\lambda$.
Also, by Lemma~\ref{lema:d}, $I_{\gamma_d}(G,X)|_{x_v=0}=\langle 1\rangle$.
Note that, if $\lambda=0$, then $m=k$, $\gamma_v=\gamma_d$, $I_{\gamma_d}(G-v, X)$ is trivial, 
and $ P_{m}^{k+i}(v) \cdot I_{\gamma_d}(G- v, X)= P_{m}^{k+i}(v)$.

Therefore, if we assume that $\lambda=0$,
\begin{eqnarray*}
I_{\gamma_d+k}(d^{k+i}(G,v),X) &=& \left\langle \left\{ P_{l}^{k+i}(v) \cdot I_{\gamma_d+k-l}(G, X)|_{x_v=0} \right\}_{l=0}^{k-1}, P_{k}^{k+i}(v)\right\rangle\\
&=& \left\langle \left\{ P_{l}^{k+i}(v) \cdot I_{\gamma_d+k-l}(G, X)|_{x_v=0} \right\}_{l=0}^{k}\right\rangle.
\end{eqnarray*}
Otherwise ($\lambda=1$), taking $l=k$ we get that $P_{k}^{k+i+1}(v) \cdot I_{\gamma_d}(G, X)|_{x_v=0}=P_{k}^{k+i+1}(v)$ and therefore
\begin{eqnarray*}
I_{\gamma_d+k}(d^{k+i+\lambda}(G,v),X)=I_{\gamma_d+k}(d^{k+i+1}(G,v),X)
&=& \left\langle \left\{ P_{l}^{k+i+1}(v) \cdot I_{\gamma_d+k-l}(G, X)|_{x_v=0} \right\}_{l=0}^{k}\right\rangle\\
&=& \left\langle \left\{ P_{l}^{k+i+\lambda}(v) \cdot I_{\gamma_d+k-l}(G, X)|_{x_v=0} \right\}_{l=0}^{k}\right\rangle.
\end{eqnarray*}
\end{proof}

When $k=1$, Theorem~\ref{teo:deq} can be reduced to the following simpler form
\[
I_{\gamma_d+1}(d^{i+1}(G,v),X)= \langle x_{v^0},x_{v^1}, \ldots, x_{v^{i+1}}, I_{\gamma_d+1}(G,X)|_{x_v=0}\rangle,
\]
for all $i\geq \lambda$, which is similar to Lemma~\ref{lema:d}.
We recall that $\lambda=\lambda(\gamma_{\mathcal P}(G-v), \gamma_{\mathcal P}(d(G,v)))$.

\begin{Remark}\label{half}
Given a fixed integer $k \geq \lambda+1$, we have that Theorem~\ref{teo:deq} implies that
\[
I_{\gamma_d+j}(d^{k}(G,v),X)=\left\langle \left\{ P_{l}^{k}(v) \cdot I_{\gamma_d+j-l}(G, X)|_{x_v=0} \right\}_{l=0}^{j} \right\rangle,
\]
for all $j$ such that $1\leq j\leq k-\lambda$.
That is, Theorem~\ref{teo:deq} does not describe all the critical ideals of $d^{k}(G,v)$.
\end{Remark}

In order to get a better understanding of Theorem~\ref{teo:deq}, we present the following example.

\begin{Example}\label{example:deq}
Let $G$ be the cycle (see Figure~\ref{fig:03}) with four vertices and sign $\sigma$ given by
\[
\sigma(e)=
\begin{cases}
-1 & \text{ if } e=v_1v_4, v_4v_3,\\
1 & \text{ otherwise.}
\end{cases}
\]
By using a computer algebra system, we can verify that $\gamma=\gamma_{\mathbb{Z}}(G)=2$, 
$\gamma_{v_1}=\gamma_{\mathbb{Z}}(G- v_1)=2$, and $\gamma_d=\gamma_{\mathbb{Z}}(d(G,v_1))=2$.
Thus $\lambda(\gamma-\gamma_{v_1},\gamma_d-\gamma)=\lambda(0,0)=0$.
Moreover, it can be checked that $I_3(G,X)=\langle x_2 + x_4, x_1-x_3,x_3x_4+2 \rangle$ and 
$I_4(G,X)=\langle x_1x_2x_3x_4+x_1x_2+x_2x_3-x_1x_4-x_3x_4-4 \rangle$.
\begin{figure}[h]
\begin{center}
\begin{tabular}{c@{\extracolsep{2cm}}c}
\multirow{9}{20mm}{
\vspace{25mm}
\begin{tikzpicture}[scale=1, line width=0.9pt]
\tikzstyle{every node}=[minimum width=4pt, inner sep=0pt, circle]
\draw (45:1) node (v1) [draw,fill=gray] {};
\draw (135:1) node (v2) [draw,fill=gray] {};
\draw (225:1) node (v3) [draw,fill=gray] {};
\draw (315:1) node (v4) [draw,fill=gray] {};
\draw[->,red, bend right] (v1) edge (v4);
\draw[->,bend right] (v4) edge (v1);
\draw[->,bend right] (v2) edge (v3);
\draw[->,bend right] (v3) edge (v2);
\draw[->,bend right] (v1) edge (v2);
\draw[->,red, bend right] (v2) edge (v1);
\draw[->,bend right] (v3) edge (v4);
\draw[->,bend right] (v4) edge (v3);
\draw (v2)+(-0.2,0.2) node () {\small $v_1$};
\draw (v3)+(-0.2,-0.2) node () {\small $v_2$};
\draw (v4)+(0.2,-0.2) node () {\small $v_3$};
\draw (v1)+(0.2,0.2) node () {\small $v_4$};
\draw[red] (0,0.6) node () {\small $-$};
\draw[red] (0.3,0) node () {\small $-$};
\end{tikzpicture}
}
&
\\
&
$
L(G, X)=
\left[\begin{array}{cccccc}
x_1 & -1 & 0 &  1 \\
 -1 & x_2 & -1 & 0  \\
 0 & -1 &  x_3 & -1 \\
-1 &   0 &  1 &  x_4 
\end{array}\right]
$
\end{tabular}
\end{center}
\caption{A signed multidigraph $G$ with four vertices and its generalized Laplacian matrix.}
\label{fig:03}
\end{figure}
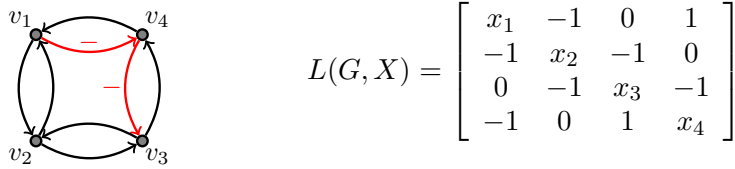

\noindent Since $I_3(G,X)|_{x_{1}=0}=\langle 2,x_3,x_2 + x_4 \rangle$, Theorem~\ref{teo:deq} implies that 
\[
I_{3}(d^{i+1}(G,v_1),X)=\left\langle P_{1}^{i+1}(v_1), I_3(G,X)|_{x_{1}=0} \right\rangle=\left\langle \{x_{1^l}\}_{l=0}^{i+1}, 2,x_3,x_2 + x_4 \right\rangle \text{ for all }i\geq 0.
\]
Also, since $I_4(G,X)|_{x_{1}=0}=\langle x_2x_3-x_3x_4-4 \rangle$, by Theorem~\ref{teo:deq} 
{\small
\begin{eqnarray*}
I_4(d^{i+2}(G,v_1),X)&\!\!\!\!=\!\!\!\!& \left\langle P_{2}^{i+2}(v_1), P_{1}^{i+2}(v_1)\cdot I_3(G,X)|_{x_{1}=0}, I_4(G,X)|_{x_{1}=0}\right\rangle\\
&\!\!\!\!=\!\!\!\!& \langle \{x_{1^l}x_{1^{l'}}\}_{0\leq l<l'\leq i+2},\{2x_{1^l}\}_{l=0}^{i+2}, \{x_{1^l}x_3 \}_{l=0}^{i+2},
\{x_{1^l}(x_2\!+\!x_4)\}_{l=0}^{i+2}, x_2x_3\!-\!x_3x_4\!-\!4 \rangle \text{ for all }i\geq 0.
\end{eqnarray*}
}
Finally, since $I_{j}(G,X)=\langle 0 \rangle$ for all $j\geq 5$,
{\small
\begin{eqnarray*}
I_{k+2}(d^{k+i}(G,v_1),X)&\!\!\!\!=\!\!\!\!& \langle P_{k}^{k+i}(v_1), P_{k-1}^{k+i}(v_1)\cdot I_{3}(G,X)|_{x_{1}=0}, P_{k-2}^{k+i}(v_1)\cdot I_{4}(G,X)|_{x_{1}=0}\rangle\\
&\!\!\!\!=\!\!\!\!& \langle P_{k}^{k+i}(v_1),\{2, x_3,x_2\!+\!x_4\}\cdot P_{k-1}^{k+i}(v_1), (x_2x_3\!-\!x_3x_4\!-\!4)\cdot P_{k-2}^{k+i}(v_1)\rangle \text{ for all }i\geq 0,k\geq 1.
\end{eqnarray*}
}
Moreover, the reader can check that 
\begin{eqnarray*}
I_4(d(G,v_1), X)&=&\langle x_{v_1^0} (x_2+x_4), x_{1^1} (x_2+x_4), x_{1^0} (x_3x_4+2),  x_2x_3-x_3x_4-4, \\
&& x_{1^0} x_{1^1}x_4 +2x_{1^0}+2x_{1^1}, x_{1^0}x_3 + x_{1^1} x_3-x_{1^0} x_{1^1} \rangle\\
& \neq & \langle P_{2}^{i+1}(v_1), P_{1}^{i+1}(v_1)\cdot I_3(G,X)|_{x_{1}=0}, I_4(G,X)|_{x_{1}=0} \rangle.
\end{eqnarray*}

That is, Theorem~\ref{teo:deq} cannot be improved.
\end{Example}

The diagonal entries of twin vertices in the Laplacian matrix of $d^k(G,v)$ are equal (twin vertices have the same degree).
Therefore, an important case of the critical ideals of $d^k(G,v)$ is given by considering the same variable associated to duplicated vertices.
In this case Theorem~\ref{teo:deq} reduces to the following form:
If $\gamma_{\mathcal P}(G) \geq 2$ and $x_v$ is the variable associated to the twins of $v$, then
\[
I_{\gamma_d+k}(d^{k+\lambda+i}(G,v),X)= \left\langle  \left\{ x_v^l \cdot I_{\gamma_d+k-l}(G, X)|_{x_v=0} \right\}_{l=0}^{k}\right\rangle 
\]
for all $k\geq 1$ and $i\geq 0$. 


\subsection{The critical ideals of the replication of vertices}
We now give the description of the critical ideals of $r^k(G,v)$.
This part is structured similarly to the part of the critical ideals of $d^k(G,v)$.
Given a subset $S$ of the natural numbers and a vertex $v\in V(G)$, let 
\[
\widetilde{P}_{l}^{S}(v)=\{\prod_{c\in C}x_{v^c}+1\, : \, C\in \binom{S}{l}\}.
\]
For convention $\widetilde{P}_{0}^S(v)=\{1\}$.
Also, for simplicity $\widetilde{P}_{l}^{\{0\}\cup [k]}(v)$ will be denoted by $\widetilde{P}_{l}^{k}(v)$.
Note that $I_{l+1}(K_k,X)=\langle \widetilde{P}^k_{l}(v)\rangle$ for all $1\leq l\leq k-2$, where $K_k$ is the complete graph with $k$ vertices.
We will use similar arguments to those used in the proof of Lemma~\ref{lema:gend}.

\begin{Lemma}\label{lema:genr}
Let $G$ be a signed multidigraph with $n\geq 2$ vertices, and $v\in V(G)$. 
If $k,j\geq 1$ and $m=\min(k,j-1)$, then
{
\begin{eqnarray*}
I_j(r^k(G,v),X)&=&\left\langle\{\widetilde{P}_{l}^{k}(v) \cdot I_{j-l}(G,X)|_{x_v=-1}\}_{l=0}^{m-1}, \widetilde{P}_{m}^{k}(v) \cdot I_{j-m}(G\!-\!v,X),\right.\\
&& \left.\widetilde{P}_{m}^{k}(v)\cdot {\rm minors}_{j-m}({\bf a},L(G\!-\!v,X)), \widetilde{P}_{m}^{k}(v)\cdot {\rm minors}_{j-m}(L(G\!-\!v,X), {\bf b}),\widetilde{S}_j^k\right\rangle
\end{eqnarray*}
}
for all $1\leq j\leq n+k$, where $\widetilde{S}_j^k$ is equal to 
\[
\left\{\prod_{s=1}^{j}(x_{v^{l_s}}\!+\!1)\!-\!\sum_{s=1}^{j}\prod_{t\neq s}(x_{v^{l_t}}\!+\!1): 0\leq l_1\!<\cdots<l_j\leq k \right\},
\]
when $j\leq k+1$, and equal to 
\[
\left\{ \det(Q)\cdot\prod_{t=0}^{k}(x_{v^{t}}\!+\!1)\!+\!\det(J(-1,{\bf a'};Q,{\bf b'}))\sum_{t=0}^{k}\prod_{s\neq t}(x_{v^{s}}\!+\!1): J(x_v,{\bf a'};Q,{\bf b'})\in M_{j\!-\!k}(L(G,X)) \right\},
\]
when $j> k+1$.
\end{Lemma}
Since the proof of this lemma is technical and very similar to those arguments given in previous proofs,
it was included at the end of this section.

We now give a similar result to Theorem~\ref{teo:deq} for the replication of vertices.
\begin{Theorem}\label{teo:req}
Let $G$ be a signed multidigraph, $v\in V(G)$, $\gamma=\gamma_{\mathcal P}(G)$, $\gamma_r=\gamma_{\mathcal P}(r(G,v))$, $\gamma_v=\gamma_{\mathcal P}(G- v)$, 
and $\lambda=\lambda(\gamma_v,\gamma_r)$.
If $\gamma\geq 2$, then $0\leq \gamma_r-\gamma_v\leq 2$ and
\[
I_{\gamma_r+k}(r^{k+\lambda+i}(G,v),X)= \left\langle \left\{ \widetilde{P}_{l}^{k+\lambda+i}(v) \cdot I_{\gamma_r+k-l}(G, X)|_{x_v=-1} \right\}_{l=0}^{k}\right\rangle, 
\]
for all $k\geq 1$ and $i\geq 0$.
\end{Theorem}
The proof follows similar by arguments to those used in Theorem~\ref{teo:deq}.
\begin{proof}
First, since $I_j(G,X)|_{x_v=-1}\subseteq I_{j-2}(G-v, X)$, Lemma~\ref{lema:r} implies that $0\leq \gamma_r-\gamma_v\leq 2$.
Now, applying Lemma~\ref{lema:genr} with $j=\gamma_r+k$ and $k=k+\lambda+i$ we have that $I_{\gamma_r+k}(r^{k+\lambda+i}(G,v),X)$ is equal to
\begin{gather*}
\left\langle \left\{ \widetilde{P}_{l}^{k+\lambda+i}(v) \cdot I_{\gamma_r+k-l}(G, X)|_{x_v=-1} \right\}_{l=0}^{m-1}, 
\widetilde{P}_{m}^{k+\lambda+i}(v) \cdot I_{\gamma_r+k-m}(G-v, X),\right.\\
\left.\widetilde{P}_{m}^{k+\lambda+i}(v)\cdot {\rm minors}_{\gamma_r+k-m}({\bf a},L(G-v, X)),
\widetilde{P}_{m}^{k+\lambda+i}(v)\cdot {\rm minors}_{\gamma_r-k-m}(L(G-v, X), {\bf b}), \widetilde{S}_{\gamma_r+k}^{k+\lambda+i}(G,v)\right\rangle,
\end{gather*}
where $m={\rm min}(k+\lambda+i, \gamma_r+k-1)=k+{\rm min}(\lambda+i, \gamma_r-1)\geq k+\lambda$.

By Lemma~\ref{lema:r}, $I_{\gamma_r}(G,X)|_{x_v=-1}=\langle 1\rangle$.
The rest follows in a similar way to the proof of Theorem~\ref{teo:deq}.
\end{proof}
We now show an example in order to understand Theorem~\ref{teo:req}.

\begin{Example}\label{example:req}
Let $G$ be the signed multidigraph given in Figure~\ref{figure:Replication}.
\begin{figure}[h]
\begin{center}
\begin{tabular}{c@{\extracolsep{2cm}}c}
\multirow{9}{3cm}{
\begin{tikzpicture}[line width=1pt, scale=1]
\tikzstyle{every node}=[inner sep=0pt, minimum width=4pt] 
\draw (0,0) node (v1) [draw, circle, fill=gray, label = right:{\small $v_1$}] {};
\draw (1,1) node (v2) [draw, circle, fill=gray, label = right:{\small $v_2$}] {};
\draw (1,-1) node (v3) [draw, circle, fill=gray, label = right:{\small $v_3$}] {};
\draw (-2,1) node (v4) [draw, circle, fill=gray, label = left:{\small $v_4$}] {};
\draw (-1,0) node (v5) [draw, circle, fill=gray, label = left:{\small $v_5$}] {};
\draw (-2,-1) node (v6) [draw, circle, fill=gray, label = left:{\small $v_6$}] {};
\draw[red] (1.2,0) node () {\small $-$};
\path[<-] (v1) edge (v2) edge (v3) edge[bend right] (v4) edge (v5) edge[bend left] (v6);
\draw (v4) -- (v5) -- (v6) -- (v4);
\draw[red] (v2) -- (v3);
\end{tikzpicture}
}
&\\
&
$
L(G, X)=
\left[\begin{array}{cccccc}
x_1 &     0 &      0 &      0 &     0 &     0 \\
 -1   & x_2 &      1 &      0 &     0 &     0 \\
 -1   &     1 &  x_3 &      0 &     0 &     0 \\
 -1   &     0 &      0 &  x_4 &    -1 &    -1 \\
 -1   &     0 &      0 &     -1 & x_5 &    -1 \\
 -1   &     0 &      0 &     -1 &     -1 & x_6 
\end{array}\right]
$\\
& \\
\end{tabular}
\end{center}
\caption{A graph $G$ with six vertices and its generalized Laplacian matrix.}
\label{figure:Replication}
\end{figure}
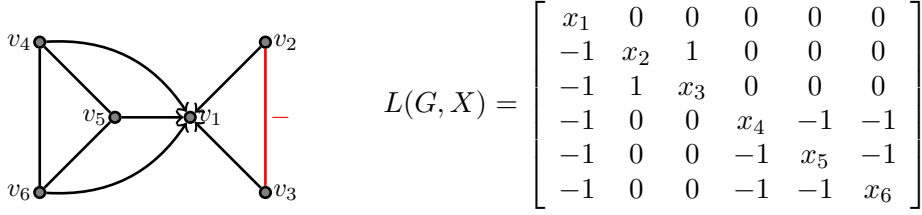

By using a computer algebra system, we have that $\gamma_{\mathbb{Z}}(G)=\gamma_{\mathbb{Z}}(G- v_1)=2$ and $\gamma_{\mathbb{Z}}(r(G,v_1))=3$.
Thus $\gamma_r-\gamma=1$ and $\lambda(\gamma-\gamma_v,\gamma_r-\gamma)=1$.
Also, it can be calculated that $I_{4}(G,X)|_{x_1=-1} = \langle  x_4+1, x_5+1, x_6+1, x_2x_3-1 \rangle$, 
\[
I_{5}(G,X)|_{x_1=-1} = \langle  (x_4+1)\cdot (x_2x_3-1), (x_5+1)\cdot (x_2x_3-1), (x_6+1)\cdot (x_2x_3-1), x_4x_5x_6-x_4-x_5-x_6-2 \rangle,
\]
and $I_{6}(G,X)|_{x_1=-1} = \langle  (x_2x_3-1)\cdot (x_4x_5x_6-x_4-x_5-x_6-2) \rangle$.
Then Theorem~\ref{teo:req} implies
\begin{eqnarray*}
I_{4}(r^{i+2}(G,v_1),X) &=& \langle  \{ x_{1^l}+1\}_{0\leq l\leq i+2}, I_{4}(G,X)|_{x_1=-1} \rangle\\
&=&\langle  \{ x_{1^l}+1\}_{0\leq l\leq i+2}, x_4+1, x_5+1, x_6+1, x_2x_3-1 \rangle,
\end{eqnarray*}
for all $i\geq 0$.
Also $I_{5}(r^{i+3}(G,v_1),X)$ is equal to
\[
\left\langle  \left\{ (x_{1^l}+1)(x_{1^{l'}}+1)\right\}_{0\leq l<l'\leq i+3}, I_{5}(G,X)|_{x_1=-1}, \left\{ (x_{1^l}+1)\cdot I_{4}(G,X)|_{x_1=-1}\right\}_{0\leq l\leq i+3}  \right\rangle
\]
for all $i\geq 0$.
Finally, $I_{k+3}(r^{k+i+1}(G,v_1),X)$ is equal to
{
\[
\left\langle \widetilde{P}_{k}^{k+i+1}(v_1), \widetilde{P}_{k-1}^{k+i+1}(v_1) \cdot I_{4}(G, X)|_{x_v=-1}, \widetilde{P}_{k-2}^{k+i+1}(v_1) \cdot I_{5}(G, X)|_{x_v=-1},  \widetilde{P}_{k-3}^{k+i+1}(v_1) \cdot I_{6}(G, X)|_{x_v=-1} \right\rangle,
\]
}
for all $k\geq 3$ and $i\geq 0$.
On the other hand, it can be check that $I_{4}(r(G,v_1),X)$ is equal to 
\[
\langle  \{ (x_{1^l}+1)(x_{l'}-1)\}_{0\leq l\leq 1, 2\leq l' \leq 3}, x_4+1, x_5+1, x_6+1, x_2x_3-1, x_{1} x_{1^1} -1 \rangle,
\]
which is different from $\langle  x_{1}+1, x_{1^1}+1 , x_4+1, x_5+1, x_6+1, x_2x_3-1 \rangle$.
Thus Theorem~\ref{teo:req} can not be improved.
\end{Example}

\begin{Remark}
Note that $I_i(K_{k+1},X)=\langle \widetilde{P}_{i-1}^{k}(v)\rangle$ for all $1\leq i\leq k$, see~\cite[Theorem 3.16]{critical}.
Moreover, if $m={\rm min}(k,j-1)$, then 
\[
I_j(r^k(G,v),X)|_{x_{v}=-1}=\left\langle \left\{ \widetilde{P}_{l}^{[k]}(v) \cdot I_{j-l}(G, X)|_{x_v=-1} \right\}_{l=0}^{m} \right\rangle. 
\]
That is, $I_j(d^k(G,v),X)|_{x_{v}=-1}$ behaves almost equal to as $I_j(K_{k+1}\sqcup G, X)$.
\end{Remark}

In a similar way to Theorem~\ref{teo:deq}, when we equal all the variables associated to the replicated vertices of $v$,
we get that Theorem~\ref{teo:req} takes the following form:
If $\gamma_{\mathcal P}(G) \geq 2$ and $x_v$ is the variable associated to all the twins of $v$, then
\[
I_{\gamma_r+k}(d^{k+\lambda+i}(G,v),X)= \left\langle  \left\{ (x_v^l+1) \cdot I_{\gamma_r+k-l}(G, X)|_{x_v=0} \right\}_{l=0}^{k}\right\rangle 
\]
for all $k\geq 1$ and $i\geq 0$. 

Examples~\ref{example:deq} and~\ref{example:req} show that the results obtained in this article are tight.
By using Theorems~\ref{teo:deq} and~\ref{teo:req} we can not determine all the critical ideals of the graph $G^{\bf d}$ for some 
${\bf d}\in \mathbb{Z}^{V(G)}$ in terms of the critical ideals of $G$.
However, there exist some special cases in which we can determine their critical ideals using very similar ideas.
For instance, in the following subsection we present the case of a complete bipartite graph.

\subsection{Critical ideals of the complete bipartite graph}\label{Sbipartite}
Given $m\geq n\geq 1$, let $K_{n,m}$ be the complete bipartite graph with bipartition $(U,V)$ such that $U$ contains $n$ vertices and $V$ contains $m$ vertices.
If $K_2$ is the complete graph with two vertices $v_1$ and $v_2$, then it is clear that $K_{n,m}=K_2^{(n-1,m-1)}$.
Now, given $0\leq j\leq n-1$, let
\[
\sigma_{j,n}(v)=
\begin{cases}
\sum_{r=1}^n\prod_{s\neq r} x_{v^s} & \text{ if } j=n-1,\\
P_{j}^{n-1}(v) & \text{ otherwise}.
\end{cases}
\]
\begin{Theorem}\label{Tbipartite} 
If $m\geq n\geq 2$, then
{\small
\[
I_j(K_{n, m}, X)=
\begin{cases}
\langle \{\sigma_{r,n}(v_1)\cdot \sigma_{s,m}(v_2): r+s=j-2, (0,0)\leq (r,s)\leq (n-1,m-1)\}\rangle & \text{ if } 2\leq j\leq n+m-2,\\
\langle \sigma_{n-1,n}(v_1)\cdot \sigma_{m-2,m}(v_2), \sigma_{n-2,n}(v_1)\cdot \sigma_{m-1,m}(v_2), P_{n-1}^{n-1}(v_1)\cdot P_{m-1}^{m-1}(v_2)\}\rangle \!\!\!\!& \text{ if } j=n+m-1,\\
\langle \prod_{r=1}^n x_{1^r} \cdot \prod_{s=1}^m x_{2^s}-\sigma_{n-1,n}(v_1)\cdot \sigma_{m-1,m}(v_2)\rangle & \text{ if } j=n+m.
\end{cases}
\]
}
\end{Theorem}

The results obtained here can be used to determine a big part of the critical ideals of the complete bipartite graph.
Since $K_{n, 2}=d^{n-2}(K_{2, 2},v_{1^1})$ and
\[
I_j(K_{2, 2}, X)=
\begin{cases}
\langle 1\rangle& \text{ if }j=1,2,\\
\langle x_1+x_{1^1},x_2+x_{2^1}, x_1x_2\rangle& \text{ if }j=3,\\
\langle x_1x_{1^1}x_2x_{2^1}-x_1x_2-x_1x_{2^1}-x_{1^1}x_2-x_{1^1}x_{2^1}\rangle & \text{ if }j=4,
\end{cases}
\]
$\gamma_v=\gamma_{\mathbb{Z}}(K_{1,2})=2=\gamma_{\mathbb{Z}}(K_{3,2})=\gamma_d$. 
Thus $\lambda=0$, $I_3(K_{2, 2}, X)|_{x_{1^1}=0}=\langle x_1,x_{2}+x_{2^1}\rangle$ and $I_4(K_{2, 2}, X)|_{x_{1^1}=0}=\langle x_1\cdot(x_{2}+x_{2^1}) \rangle$.
Applying~Theorem~\ref{teo:deq} with $k=j-2$, $i=n-j\geq 0$ we get that
\begin{eqnarray*}
I_j(K_{n, 2}, X)&=&I_{2+(j-2)}(d^{n-2}(K_{2, 2},v_{1^1}), X)\\
&=&\langle P_{j-2}^{n-2}(v_{1^1}), P_{j-3}^{n-2}(v_{1^1})\cdot I_3(K_{2, 2}, X)|_{x_{1^1}=0}, P_{j-4}^{n-2}(v_{1^1})\cdot I_4(K_{2, 2}, X)|_{x_{1^1}=0}\rangle\\
&=&\langle \{\sigma_{j-2,n}(v_1)\cdot \sigma_{0,2}(v_2), \sigma_{j-3,n}(v_1)\cdot \sigma_{1,2}(v_2)\}\rangle 
\end{eqnarray*}
for all $3\leq j\leq n+m-2$.
In a similar way we can use the critical ideals of $K_{n,2}$ to determine the first $m$ critical ideals of $K_{n,m}$.
That is, we can determine more than one half of the critical ideals of $K_{n,m}$.
The remaining critical ideals can be determined using similar, but more specific techniques.
In a more general setting~Theorem~\ref{teo:deq} can be used to determine a part of the critical ideals of the complete multipartite graphs.

Moreover, Theorems~\ref{teo:deq} and~\ref{teo:req} can be improved in the special case when several vertices are duplicated and replicated simultaneously, 
which allows us to describe almost completely the critical ideals of complete multipartite graphs and threshold graphs.


\subsection{Proofs of Lemmas~\ref{lema:gend} and~\ref{lema:genr}}

\mbox{}

\textbf{Proof of Lemma \ref{lema:gend}:}
The generalized Laplacian matrix $L(d^k(G,v),X)$ of $d^k(G,v)$ is equal to 
\[
J({\rm diag}(x_{v^0},..., x_{v^k}),{\bf a};L(G-v,X),{\bf b}),
\]
for some ${\bf a},{\bf b} \in \mathcal{P}^{n-1}$.
Let $\mathcal{I,I'}\subseteq [n+k]$ be two sets of size $j$, $h=|\mathcal{I}\cap [k+1]|$, $h'=|\mathcal{I'}\cap [k+1]|$, and 
\[
m_{\mathcal{I,I'}}={\rm det}(L(d^k(G,v),X)[\mathcal{I,I'}]).
\]
Clearly $0\leq h,h'\leq m+1$.
If $h,h'=0$, then $m_{\mathcal{I,I'}}\in {\rm minors}_{j}(L(G-v,X))$ and $m_{\mathcal{I,I'}}\in I_{j}(G-v,X)$. 
Now assume that $h=0$.
If $h'\geq 2$, then two columns of $L(d^k(G,v),X)[\mathcal{I,I'}]$ are equal, and $m_{\mathcal{I,I'}}=0$.
Also, if $h'=1$, then $m_{\mathcal{I,I'}}\in {\rm minors}_j({\bf a}, L(G-v,X))$. 
We can use similar arguments when $h'=0$.
Thus, we assume that $h,h'\geq 1$.

Now by Lemma~\ref{lema:det1} we have that 
\[
m_\mathcal{I,I'}=
\begin{cases}
0 & \text{ if } |h-h'|\geq 2,\\
{\rm det}\left[\begin{array}{cc} P&1\end{array}\right]\cdot {\rm det}\left[\begin{array}{c}{\bf b}'\\Q\end{array}\right]& \text{ if } h-h'=1,\\
{\rm det}\left[\begin{array}{c}P\\{\bf 1}\end{array}\right]\cdot{\rm det}\left[\begin{array}{cc}{\bf a}'^T &Q\end{array}\right] & \text{ if } h'-h=1,\\
{\rm det}(P)\cdot {\rm det}(Q)-{\rm det}(J(P,{\bf 1};0,{\bf 1})) \cdot {\rm det}(J(0,{\bf a'};Q,{\bf b'})) & \text{ if } h=h',
\end{cases}
\]
for some submatrix $P$ of ${\rm diag}(x_{v^0},..., x_{v^k})$, some submatrix $Q$ of $L(G- v, X)$, and some subvectors ${\bf a}'$ of ${\bf a}$ and ${\bf b}'$ of ${\bf b}$.
Clearly, ${\rm det}\left[\begin{array}{cc} P& {\bf 1}\end{array}\right]\neq 0$ if and only if (up to row and column permutations)
\[
P=\left[\begin{array}{c} {\rm diag}(x_{v^{i_1}}, \ldots, x_{v^{i_{h'}}})\\ {\bf0}\end{array}\right].
\]
If $h-h'=1$, then $m_\mathcal{I,I'}\in P_{h'}^{k}(v)\cdot {\rm minors}_{j-h'}(L(G- v, X),{\bf b})\subsetneq P_{h'}^{k}(v)\cdot I_{j-h'}(G,X)|_{x_v=0}$, for all $1\leq h'\leq m$.
Similarly, if $h'-h=1$, then $m_\mathcal{I,I'}\in P_{h}^{k}(v)\cdot {\rm minors}_{j-h}({\bf a}, L(G- v, X))\subsetneq P_{h}^{k}(v)\cdot I_{j-h}(G,X)|_{x_v=0}$, for all $1\leq h\leq m$.
On the other hand, if $h=h'$ we have the following cases:

\noindent Case I: If $P$ has at least two zero rows, then ${\rm det}(P)=0$, ${\rm det}(J(P,{\bf 1};0,{\bf 1})) =0$, and $m_\mathcal{I,I'}=0$.

\noindent Case II:  If $P$ has only one zero row, then ${\rm det}(P)=0$, 
${\rm det}(J(P,{\bf 1};0,{\bf 1})) =\prod_{t=1}^{h-1} x_{v^{i_t}}$, and
\[
m_\mathcal{I,I'}=\prod_{t=1}^{h-1} x_{v^{i_t}}\cdot {\rm det}(J(0,{\bf a}';Q,{\bf b}')),
\] 
for some $(j-h+1)\times (j-h+1)$-submatrix $J(0,{\bf a}';Q,{\bf b}')$ of $L(G, X)|_{x_v=0}$.
Thus $m_\mathcal{I,I'}\in P_{h-1}^{k}(v)\cdot {\rm minors}_{j-h+1}({\bf a}, L(G- v, X),{\bf b})\subsetneq P_{h-1}^{k}(v)\cdot I_{j-h+1}(G,X)|_{x_v=0}$, for all $2\leq h \leq m-1$.

\noindent Case III: If $P$ has no zero row, then 
\[
m_\mathcal{I,I'}=
\begin{cases} 
\prod_{t=1}^h x_{v^{i_t}}\cdot {\rm det}(Q)+\sum_{t=1}^h \prod_{s\neq t} x_{v^{i_s}}\cdot {\rm det}(J(0,{\bf a'};Q,{\bf b'}))
& \text{ if } h< j,\\
\prod_{t=1}^h x_{v^{i_t}} & \text{ if } h= j,
\end{cases}
\]
for some $(j-h+1)\times(j-h+1)$-submatrix $J(0,{\bf a'};Q,{\bf b'})$ of $L(G, X)|_{x_v=0}$, and for all $1\leq h \leq m$.
Moreover, since 
\[
\sum_{t=1}^h \prod_{s\neq t} x_{v^{i_s}}\cdot {\rm det}(J(0,{\bf a'};Q,{\bf b'}))
\in \langle P_{h-1}^{k}(v)\cdot {\rm minors}_{j-h+1}({\bf a}, L(G-v,X),{\bf b})\rangle
\]
and $\prod_{t=1}^h x_{v^{i_t}}\cdot {\rm det}(Q)=m_\mathcal{I,I'}-\sum_{t=1}^h \prod_{s\neq t} x_{v^{i_s}}\cdot {\rm det}(J(0,{\bf a'};Q,{\bf b'}))
\in \langle P_{h}^{k}(v)\cdot{\rm minors}_{j-h}(L(G- v, X)) \rangle\subsetneq P_{h}^{k}(v)\cdot I_{j-h}(G,X)|_{x_v=0}$
for all $0\leq h\leq m-1$, we get the result.

\bigskip

\textbf{Proof of Lemma \ref{lema:genr}:}
Let $\mathcal{I,I'}\subseteq [n+k]$ be two sets of size $j$, $h=|\mathcal{I}\cap [k+1]|$ and $h'=|\mathcal{I'}\cap [k+1]|$.
Clearly $0\leq h,h'\leq m+1$ and $L(r^k(G,v),X)=J(L(K_{k+1},X), {\bf a};L(G- v, X), {\bf b})$ for some $\textbf{a},\textbf{b}\in\mathcal{P}^{n-1}$.
Let $m_\mathcal{I,I'}=\det(L(r^k(G,v),X)[\mathcal{I,I'}])$.

We can use the same arguments used in the proof of Lemma~\ref{lema:gend} for the case when $h=0$ or $h'=0$.
On the other hand, by Lemma~\ref{lema:det1}
\[
m_\mathcal{I,I'}=\left\{\begin{array}{ll}
0&\textrm{if }|h-h'|>2,\\
\det\left[\begin{array}{cc} P& {\bf 1}^T\end{array}\right]\det\left[\begin{array}{c} \textbf{b}'\\ Q\end{array}\right]&\textrm{if }h-h'=1,\\
\det\left[\begin{array}{c} P\\ {\bf 1}\end{array}\right]\det\left[\begin{array}{cc} \textbf{a}'^T& Q\end{array}\right]&\textrm{if }h'-h=1,\\
\det(P)\det(Q)-\det\left[\begin{array}{cc} P& {\bf 1}^T\\{\bf 1}&0\end{array}\right]\det\left[\begin{array}{cc} 0& \textbf{b}'\\ \textbf{a}'^T&Q\end{array}\right]&\textrm{if }h=h',
\end{array}\right.
\]
where $P$ is a submatrix of $L(K_{k+1},X)$, $Q$ is a submatrix of $L(G- v, X)$, $\textbf{a}'$ is a subvector of $\textbf{a}$ and $\textbf{b}'$ is a subvector of $\textbf{b}$. 
If $h-h'=1$ then $\det\left[\begin{array}{cc} P& 1\end{array}\right]\neq 0$ if and only if (up to row and column permutations)  
\[
P=\left(\begin{array}{cccc} x_{v^{i_1}}& & -1&-1\\ &\ddots& \\ -1& &x_{v^{i_{h'}}}& -1\end{array}\right)^T
\]
for some $0\leq l_1<\cdots< l_{h'}\leq k$.  
Since $\det\left[\begin{array}{cc} P& {\bf 1}^T\end{array}\right]=\prod_{s=1}^{h'} (x_{v^{i_s}}+1)$, 
$m_\mathcal{I,I'}\in \widetilde{P}_{h'}^{k}(v) \cdot {\rm minors}_{j-h'}(L(G-v, X),{\bf b})\subsetneq \widetilde{P}_{h'}^{k}(v) \cdot I_{j-h'}(G,X)|_{x_v=-1}$.
In a similar way, if $h'-h=1$, then  $m_\mathcal{I,I'}\in\widetilde{P}_{h}^{k}(v)\cdot I_{j-h}(G,X)|_{x_v=-1}$.

Now assume that $h=h'$.
If $P$ has two rows equal to $-{\bf 1}$, then $m_\mathcal{I,I'}=0$. 
Let
\[
R=\left(\begin{array}{ccc} x_{v^{l_1}}& & -1\\ &\ddots& \\ -1& &x_{v^{l_h}}\end{array}\right) 
\]
where $0\leq l_1<\cdots<l_h\leq k$. 
If $P$ has only a row equal to $-{\bf 1}$, then $P$ is equal to (up to row and column permutations) $R|_{x_{v^{l_h}}=-1}$.
Since ${\rm det}(R|_{x_{v^{l_h}}=-1})=-\prod_{s=1}^{h-1}(x_{v^{l_s}}+1)$ and $\det(J(R|_{x_{v^{l_h}}=-1},{\bf 1};0,{\bf 1}))=-\prod_{s=1}^{h-1}(x_{v^{l_s}}+1)$,
\[
m_\mathcal{I,I'}=\left(\det(J(0,{\bf a'};Q,{\bf b'}))-\det(Q)\right)\prod_{s=1}^{h-1}(x_{v^{l_s}}+1) = \det(J(-1,{\bf a'};Q,{\bf b'}))\prod_{s=1}^{h-1}(x_{v^{l_s}}+1),
\text{ for all } 1\leq h \leq m.
\]
Thus $m_\mathcal{I,I'}\in\langle \widetilde{P}_{h-1}^{k}(v)\cdot I_{j-h+1}(G,X)|_{x_v=-1}\rangle$.
Finally, if $P$ has no row equal to $-{\bf 1}$, then $P$ is equal to (up to row and column permutations) to $R$.
Since ${\rm det}(R)=\prod_{s=1}^{h}(x_{v^{l_s}}+1)-\sum_{s=1}^{h}\prod_{t\neq s}(x_{v^{l_t}}+1)$ (see~\cite[Theorem 3.15]{critical}) 
and $\det(J(R,{\bf 1};0,{\bf 1})) = -\sum_{s=1}^{h}\prod_{t\neq s}(x_{v^{l_t}}+1)$,
\begin{eqnarray*}
m_\mathcal{I,I'}&=& \det(Q)\cdot \prod_{s=1}^{h}(x_{v^{l_s}}+1)+\left(\det(J(0,{\bf a'};Q,{\bf b'}))-\det(Q)\right)\cdot\sum_{s=1}^{h}\prod_{t\neq s}(x_{v^{l_t}}+1)\\
&=& \det(Q)\cdot\prod_{s=1}^{h}(x_{v^{l_s}}+1)+\det(J(-1,{\bf a'};Q,{\bf b'}))\cdot\sum_{s=1}^{h}\prod_{t\neq s}(x_{v^{l_t}}+1),\text{ for all }1\leq h \leq m.
\end{eqnarray*}
Since 
$\det(Q)\cdot\prod_{s=1}^{h}(x_{v^{l_s}}+1)=m_\mathcal{I,I'}-\det(J(-1,{\bf a'};Q,{\bf b'}))\cdot\sum_{s=1}^{h}\left(\prod_{t\neq s}(x_{v^{l_t}}+1)\right) 
\in \widetilde{P}_{h}^{k}(v)\cdot I_{j-h}(G,X)|_{x_v=-1}$ we get the result.


\section{Acknowledgement}

The authors would like to thank the anonymous referee for his helpful comments.
	

\bibliographystyle{abbrv}
\bibliography{Biblio}

\end{document}